\documentclass{amsart}
\usepackage{graphicx}
\usepackage{graphicx,esint}
\usepackage{amssymb,amscd,amsthm,amsxtra}
\usepackage{latexsym}
\usepackage{epsfig}

\newtheorem{thm}{Theorem}[section]
\newtheorem{cor}[thm]{Corollary}
\newtheorem{lem}[thm]{Lemma}

\theoremstyle{definition}
\newtheorem{defn}[thm]{Definition}
\theoremstyle{remark}

\numberwithin{equation}{section}

\newcommand{\R}{\mathbb R}
\newcommand{\be}{\begin{equation}}
\newcommand{\ee}{\end{equation}}

\newcommand{\ep}{\eps}

\newcommand{\eps}{\varepsilon}

\newcommand{\p}{\partial}

\newcommand{\comment}[1]{}

\begin{document}

\title[Two phase problems for distributed sources]{Two-phase problems with distributed source: regularity of the free boundary}%
\author{D. De Silva}
\address{Department of Mathematics, Barnard College, Columbia University, New York, NY 10027}
\email{\tt  desilva@math.columbia.edu}
\author{F. Ferrari}
\address{Dipartimento di Matematica dell'Universit\`a di Bologna, Piazza di Porta S. Donato, 5, 40126 Bologna, Italy.}
\email{\tt fausto.ferrari@unibo.it}
\author{S. Salsa}
\address{Dipartimento di Matematica del Politecnico di Milano, Piazza Leonardo da Vinci, 32, 20133 Milano, Italy.}
\email{\tt sandro.salsa@polimi.it }

\thanks{ D.~D.~ and F.~ F.~  are supported by the ERC starting grant project 2011 EPSILON (Elliptic PDEs and Symmetry of Interfaces and Layers for Odd Nonlinearities). F.~F.~is supported by  Miur, Italy, by University of Bologna, Italy.
S.~S.~ is supported by Miur Grant, Geometric Properties of Nonlinear Diffusion Problems.  F.~F.~\ wishes to thank the Department of Mathematics of Columbia University, New York, for the  kind hospitality.}

\begin{abstract}We investigate the regularity of the free boundary for a general class of two-phase free boundary problems with non-zero right hand side. We prove that Lipschitz or flat free boundaries are $C^{1,\gamma}$. In particular, viscosity solutions are indeed classical.\end{abstract}

\maketitle

\section{Introduction and main results}

In this paper we consider two phase free boundary problems governed by
uniformly elliptic equations with distributed sources. Our purpose is to
investigate the regularity of the free boundary under additional hypotheses
such as flatness or Lipschitz continuity. A model problem we have in mind is
the following:
\begin{equation}  \label{fb}
\left\{
\begin{array}{ll}
\Delta u = f, & \hbox{in $\Omega^+(u) \cup \Omega^-(u),$} \\
\  &  \\
(u_\nu^+)^2 - (u_\nu^-)^2= 1, & \hbox{on $F(u):= \partial \Omega^+(u) \cap
\Omega.$} \\
&
\end{array}%
\right.
\end{equation}

Here, as usually for any bounded domain $\Omega \subset \mathbb{R}^n$,
\begin{equation*}
\Omega^+(u):= \{x \in \Omega : u(x)>0\}, \quad \Omega^-(u):= \{x \in \Omega
: u(x)\leq 0\}^\circ,
\end{equation*}
and $u_\nu^+$ and $u_\nu^-$ denote the normal derivatives in the inward
direction to $\Omega^+(u)$ and $\Omega^-(u)$ respectively.

Typical examples are the Prandtl-Bachelor model in fluiddynamics (see e.g.
\cite{B1,EM}), where $f=\mathbf{1}_{\Omega ^{-}\left( u\right) }$, the
characteristic function of the negative phase, or the eigenvalue problem in
magnetohydrodynamics (1,1) considered in \cite{FL}, where $f=\lambda u.$
Other examples come from limits of singular perturbation problems with
forcing term as in \cite{LW}, where the authors analyze solutions to  
\eqref{fb}, arising in the study of flame propagation with nonlocal
effects. \smallskip

The homogeneous case $f\equiv 0$ was settled in the classical works of
Caffarelli \cite{C1,C2}. A key step in these papers is the construction of a
family of continuous supconvolution deformations that act as comparison
subsolutions.

The results in \cite{C1,C2} have been widely generalized to different
classes of homogeneous elliptic problems. See for example \cite{CFS, FS1,
FS2} for linear operators, \cite{AF,F1, F2, Fe1,W1, W2} for fully nonlinear
operators and \cite{LN} for the $p$-Laplacian. All these papers follow the
guidelines of \cite{C1,C2}.

In \cite{D}, De Silva introduced a new strategy to investigate inhomogeneous
free boundary problems, motivated by a classical one
phase problem in hydrodynamic. This method has been successfully applied
in \cite{DR} to nonlocal one phase Bernoulli type problems, governed by the
fractional Laplacian. For another application of the techniques in \cite{D} see also \cite{LT}.

Here we extend the method in \cite{D} to two phase
problems to prove that flat (see below) or Lipschitz free boundaries of %
\eqref{fb} are $C^{1,\gamma}.$

In order to better emphasize the ideas involved, we first develop the regularity
theory for free boundaries of viscosity solutions to problem \eqref{fb} (see Section \ref{section2} for the
relevant definitions), and then we extend our results to a more general class of free boundary problems. For simplicity, in order to avoid the machinery of $%
L^p$-viscosity solution, we assume that $f$ is bounded in $\Omega$ and
continuous {in $\Omega^+(u) \cup \Omega^-(u)$}.  Our results may be extended to the case when
 $f$ is merely bounded measurable.
 
We remark that in view of Theorem 4.5 in \cite{CJK}, a viscosity solution to \eqref{fb} is
locally Lipschitz. In fact, as it can be easily checked, our viscosity
solutions are also weak solutions in the sense of Definition 4.4 in that
paper and both $\Delta u^{\pm} -f$ are non negative Radon measures.

We now state our first main results. Here constants depending
only on $n, \|f\|_\infty,$ and $Lip(u)$ will be called universal. 

\begin{thm}[Flatness implies $C^{1,\protect\gamma}$]
\label{flatmain1} Let $u$ be a (Lipschitz) viscosity solution to \eqref{fb}
in $B_1$. Assume that $f \in L^\infty(B_1)$ is continuous {in $B_1^+(u)
\cup B_1^-(u).$} 
There exists a universal constant $\bar \delta>0$ such that, if
\begin{equation}  \label{flat}
\{x_n \leq - \delta\} \subset B_1 \cap \{u^+(x)=0\} \subset \{x_n \leq
\delta \},
\end{equation}
with $0 \leq \delta \leq \bar \delta,$ then $F(u)$ is $C^{1,\gamma}$ in $%
B_{1/2}$.
\end{thm}

Theorem \ref{flatmain1} still holds when \eqref{flat} is replaced by
other common flatness conditions (see Subsection \ref{subsection7}).

\begin{thm}[Lipschitz implies $C^{1,\protect\gamma}$]
\label{Lipmain} Let $u$ be a (Lipschitz) viscosity solution to \eqref{fb} in $B_1$, with
$0\in F(u)$. Assume that $f \in L^\infty(B_1)$ is continuous {in $B_1^+(u)
\cup B_1^-(u).$} 
 If $F(u)$ is a Lipschitz graph in a neighborhood of $0$, then $%
F(u)$ is $C^{1,\gamma}$ in a (smaller) neighborhood of $0$.
\end{thm}

The proof of Theorem \ref{flatmain1} is based on an
improvement of flatness, obtained via a compactness argument which linearizes the
problem into a limiting one.  The key tool is a geometric Harnack inequality
that localizes the free boundary well, and allows the rigorous passage to
the limit.

The main difficulty in the analysis  comes from
the case when $u^-$ is degenerate, that is very close to zero without being
identically zero. In this case the flatness assumption does not guarantee
closeness of $u$ to an ``optimal" (two-plane) configuration. Thus one needs
to work only with the positive phase $u^+$  to balance the situation
in which $u^+$ highly predominates over $u^-$ and the case in which $u^-$ is
not too small with respect to $u^+.$

Theorem \ref{Lipmain} follows from Theorem \ref{flatmain1} and the main
result in \cite{C1}, via a blow-up argument.

Sections \ref{section2} through \ref{section6} are devoted to the proof of the theorems above. In
particular, in Section \ref{section2} we introduce the relevant definitions and some
preliminary lemmas. In Section \ref{section3} we describe the linearized problem
associated to \eqref{fb}. Section \ref{Harnacksec} is devoted to the proof of Harnack
inequality both in the non-degenerate and in the degenerate setting. In
Section \ref{section5}, we present the proof of the improvement of flatness lemmas.
Section \ref{section6} contains the proof of the  Theorem \ref{flatmain1} and Theorem \ref{Lipmain}.

From Section \ref{section8} to Section \ref {section11} we deal with more general problems of the form 
\begin{equation}  \label{fbcv}
\left \{
\begin{array}{ll}
\mathcal{L}u = f, & \hbox{in $\Omega^+(u) \cup
\Omega^-(u),$} \\
\  &  \\
u_\nu^+=G(u^-_\nu,x), & \hbox{on $F(u):= \partial \Omega^+(u) \cap \Omega,$}
\\
&
\end{array}%
\right.
\end{equation} with $f$ bounded on $\Omega$ and continuous in $\Omega^+(u) \cup \Omega^-(u)$, and $u$ Lipschitz continuous with $Lip(u)  \leq L$. Here
\begin{equation*}
\mathcal{L}=\sum_{i,j=1}^na_{ij}(x)D_{ij} + {\bf b} \cdot \nabla, \quad a_{ij}\in C^{0,\bar\gamma}(\Omega), \:\:\: {\bf b} \in C(\Omega)\cap L^\infty(\Omega),
\end{equation*} is uniformly elliptic i.e. there exist
 $0<\lambda\leq\Lambda$ such that, for every $\xi\in\mathbb{R}^n$ and
every $x\in\Omega$,\begin{equation*}
\lambda|\xi|^2\leq \sum_{i,j=1}^na_{ij}(x)\xi_i\xi_j\leq \Lambda|\xi|^2,
\end{equation*} 
and
$$G(\eta,x):[0,\infty)\times\Omega\rightarrow(0,\infty)$$
satisfies the 
following assumptions:
\begin{itemize}
\item[(H1)]
$G(\eta,\cdot)\in C^{0,\bar\gamma}(\Omega)$ uniformly in $\eta;\quad G(\cdot,x)\in C^{1,\bar\gamma}([0,L])$ for every $x\in \Omega.$

\item[(H2)]
$G'(\cdot, x)>0$ with $G(0,x)\geq\gamma_0>0$ uniformly in $x$.
\item[(H3)]
There exists $N>0$ such that $\eta^{-N}G(\eta,x)$ is strictly decreasing in $\eta$, uniformly in $x$.
\end{itemize}
\smallskip

In this framework we prove the following  main results. Here, a constant
depending (possibly) on $n,Lip(u),\lambda,\Lambda,[a_{ij}]_{C^{0,\bar\gamma}},\|{\bf b}\|_{L^
\infty}, \|f\|_{L^\infty}, [G(\eta,\cdot)]_{C^{0,\bar\gamma}},$ $\gamma_0$ and $N$ is called universal. The $C^{1,\bar\gamma}$ norm of $G(\cdot,x)$ may depend on $x$ and enters our proofs in a qualitative way only.

\begin{thm}[Flatness implies $C^{1,\protect\gamma}$]
\label{flatmain2} Let $u$ be a Lipschitz  viscosity solution to $(\ref{fbcv})$
in $B_1,$ with $Lip(u) \leq L$. Assume that $f$ is continuous {in $B_1^+(u)
\cup B_1^-(u),$} $\|f\|_{L^\infty(B_1)} \leq L$ and $G$ satisfies $(H1)$-$(H3)$. There exists a universal
constant $\bar \delta>0$ such that, if
\begin{equation*}
\{x_n \leq - \delta\} \subset B_1 \cap \{u^+(x)=0\} \subset \{x_n \leq
\delta \},
\end{equation*}
with $0 \leq \delta \leq \bar \delta,$ then $F(u)$ is $C^{1,\gamma}$ in $
B_{1/2}$.
\end{thm}

\begin{thm}[Lipschitz implies $C^{1,\protect\gamma}$]
\label{Lipmainvar} Let $u$ be a Lipschitz viscosity solution to $(\ref{fbcv})$ in $B_1$,
with $0\in F(u)$ and $Lip(u) \leq L.$ Assume that $f$ is continuous {in $B_1^+(u)
\cup B_1^-(u)$}, $\|f\|_{L^\infty(B_1)} \leq L$ and $G$ satisfies $(H1)$-$(H3)$. If $F(u)$ is a Lipschitz graph in a neighborhood of $0$, then $F(u)$ is $
C^{1,\gamma}$ in a (smaller) neighborhood of $0$.
\end{thm}

Some remarks are in order. In particular,  further extensions can be achieved with small extra effort.  Actually, there is no problem to extend our results to the case when ${\bf b}$ and $f$ are merely bounded measurable, however as we already mentioned in the prototype problem we wish to avoid too many technicalities.

 In Theorems \ref{flatmain2} and \ref{Lipmainvar} we need to assume the Lipschitz continuity of our solution
unless the operator can be put into divergence form. Indeed, in this case an
almost monotonicity formula is available (see \cite{MP}) and under the assumption $G(\eta,x)\rightarrow \infty, \text{ as } \eta\rightarrow \infty$ one can
reproduce the proof of Theorem 4.5 in \cite{CJK}, to recover the Lipschitz
continuity of a viscosity solution. Observe that then $f=f(x,u,\nabla u)$ is allowed, with $f(x,\cdot,\cdot)$ locally bounded.

\section{Compactness and localization lemmas}\label{section2}

In this section, we state basic definitions and we prove some elementary lemmas. First we need the following standard notion.

\begin{defn}Given $u, \varphi \in C(\Omega)$, we say that $\varphi$
touches $u$ by below (resp. above) at $x_0 \in \Omega$ if $u(x_0)=
\varphi(x_0),$ and
$$u(x) \geq \varphi(x) \quad (\text{resp. $u(x) \leq
\varphi(x)$}) \quad \text{in a neighborhood $O$ of $x_0$.}$$ If
this inequality is strict in $O \setminus \{x_0\}$, we say that
$\varphi$ touches $u$ strictly by below (resp. above).
\end{defn}

We refer to the usual $C$-viscosity definition of subsolution, supersolution and solution of an elliptic PDE, see e.g. \cite{CC}. 
Let us  introduce the
notion of comparison subsolution/supersolution.

\begin{defn}\label{defsub}
We say that $v \in C(\Omega)$ is a strict (comparison) subsolution (resp.
supersolution) to (\ref{fb}) in $\Omega$, if and only if $v \in C^2(\overline{\Omega^+(v) }) \cap  C^2(\overline{\Omega^-(v) })$ and the
following conditions are satisfied:
\begin{enumerate}
\item $ \Delta v  > f $ (resp. $< f $) in $\Omega^+(v) \cup \Omega^-(v)$;
\item If $x_0 \in F(v)$, then $$(v_\nu^+)^2 - (v_\nu^-)^2 >1\quad (\text{resp. $(v_\nu^+)^2 - (v_\nu^-)^2 <1$, $v_\nu^+(x_0) \neq 0$).}$$
\end{enumerate}

\end{defn}

Notice that by the implicit function theorem, according to our definition the free boundary of a comparison subsolution/supersolution is $C^2$.

Finally we can give the definition of viscosity solution to the problem (\ref{fb}).

\begin{defn}\label{defnhsol} Let $u$ be a continuous function in
$\Omega$. We say that $u$ is a viscosity solution to (\ref{fb}) in
$\Omega$, if the following conditions are satisfied:
\begin{enumerate}
\item $ \Delta u = f$ in $\Omega^+(u) \cup \Omega^-(u)$ in the
viscosity sense;
\item  Any (strict) comparison subsolution $v$ (resp. supersolution) cannot touch $u$ by below (resp. by above) at a point $x_0 \in F(v) $ (resp. $F(u)$.)
\end{enumerate}
\end{defn}

The next lemma shows that ``$\delta-$flat" viscosity solutions (in the sense of our main Theorem \ref{flatmain1}) enjoy non-degeneracy of the positive part $\delta$-away from the free boundary.
Precisely,
\begin{lem}\label{deltand}Let $u$ be a solution to \eqref{fb} in $B_2$ with  $Lip(u) \leq L$ and $\|f\|_{L^\infty} \leq L$. If
$$\{x_n \leq g(x') - \delta\} \subset \{u^+=0\} \subset \{x_n \leq g(x') + \delta\},$$ with $g$ a Lipschitz function, $Lip(g) \leq L, g(0)=0$,
then $$u(x) \geq c_0 (x_n- g(x')), \quad x \in \{x_n \geq g(x') + 2\delta\}\cap B_{\rho_0}, $$ for some $c_0, \rho_0 >0$ depending on $n,L$ as long as  $\delta \leq c_0.$
\end{lem}

\begin{proof} All constants in this proof will depend on $n,L.$

It suffices to show that our statement holds for $ \{x_n \geq g(x') + C\delta\}$ for a possibly large constant $C$. Then one can apply Harnack inequality to obtain the full statement. 

We prove the statement above at $x=de_n$ (recall that $g(0)=0$). Precisely, we want to show that $$u(de_n) \geq c_0 d, \quad d \geq C \delta.$$After rescaling, we reduce to proving that $$u(e_n) \geq c_0$$ as long as $\delta \leq 1/C$, and $\|f\|_\infty$ is sufficiently small.
Let $\gamma>0$ and $$w(x)= \frac{1}{2\gamma}(1-|x|^{-\gamma})$$ be defined on the closure of the annulus $B_2 \setminus \overline B_1$ with  $\|f\|_\infty$ small enough so that
$$\Delta w <- \|f\|\quad \textrm{on  $B_2\setminus  \overline{B}_{1}$}.$$ Let $$w_t(x) = w(x+te_n).$$ Notice that $$|\nabla w_0| < 1\quad \textrm{on $\p B_{1}.$}$$

From our flatness assumption for $t>,0$ sufficiently large (depending on the Lipschitz constant of $g$), $w_t$ is strictly above $u$.  We decrease $t$ and let $\bar t$ be the first $t$ such that $w_t$ touches $u$ by above. Since $w_{\bar t}$ is a strict supersolution to $\Delta u=f$ in 
$B_2\setminus\bar{B}_1$ the touching point $z$ can occur only on the 
$\eta:=\frac{1}{2\gamma}(1-2^{-\gamma}) $ level set in the positve phase of $u$, and $|z|\leq C = C(L).$  Since $u$ is  Lipschitz continuous, $0< u (z) = \eta \leq L d(z, F(u))$, that is a full ball around $z$ of radius  $\eta/L$ is contained in the positive phase of $u$. Thus, for $\bar \delta$ small depending on $\eta, L$ we have that $B_{\eta/2L}(z) \subset \{x_n \geq  g(x') + 2 \bar \delta\}$. Since $x_n =g(x') + 2 \bar \delta$ is Lipschitz we can connect $e_n$ and $z$ with a chain of intersecting balls included in the positive side of $u$ with radii comparable to $\eta/2L$.  The number of balls depends on $L$ . Then we can apply Harnack inequality and obtain $$u(e_n) \geq c u(z)= c_0,$$ as desired.
\end{proof}

Next, we state a compactness lemma. For its proof, we refer the reader to Section \ref{section8} where the analogue of this result for a more general class of operators and free boundary conditions is stated and proved (see Lemma \ref{compvar}).

\begin{lem} \label{compact_delta}Let $u_k$ be a sequence of viscosity solutions to \eqref{fb}  with right-hand-side $f_k$ satisfying  $\|f_k\|_{L^\infty} \leq L.$ Assume $u_k \to u^*$ uniformly on compact sets, and $\{u_k^+=0\} \to \{(u^*)^+=0\}$ in the Hausdorff distance. Then $$-L \leq \Delta u^* \leq L, \quad \text{in $\Omega^+(u^*) \cup \Omega^-(u^*)$} $$ in the viscosity sense and $u^*$ satisfies the free boundary condition  $$ ({u^*_\nu}^+)^2 - ({u^*_\nu}^-)^2= 1  \quad \hbox{on $F(u^*)$}$$ in the viscosity sense of Definition $\ref{defnhsol}.$
\end{lem}

We are now ready to re-formulate our main Theorem \ref{flatmain1} using the two lemmas above. First, we denote by $U_\beta$ the following one-dimensional function, $$U_\beta(t) = \alpha t^+ - \beta t^-, \quad  \beta \geq 0, \quad \alpha = \sqrt{1 +\beta^2},$$ where $$t^+ = \max\{t,0\}, \quad t^-= -\min\{t,0\}.$$Then $U_\beta(x)= U_\beta(x_n)$ is the so-called two-plane solution to \eqref{fb} when $f \equiv 0$.

\begin{lem}\label{normalize} Let $u$ be a solution to \eqref{fb} in $B_1$ with  $Lip(u) \leq L$ and $\|f\|_{L^\infty} \leq L$. For any $\ep >0$ there exist $\bar \delta, \bar r >0$ depending on $\ep, n,$ and  $L$ such that if \begin{equation*} \{x_n \leq - \delta\} \subset B_1 \cap \{u^+(x)=0\} \subset \{x_n \leq \delta \},\end{equation*} with $0 \leq \delta \leq \bar \delta,$ then
\be\label{conclusion_beta}\|u - U_{\beta}\|_{L^{\infty}(B_{\bar r})} \leq \eps \bar r\ee for some $0 \leq \beta \leq L.$

\end{lem}

\begin{proof} Given $\eps>0$ and $\bar r$ depending on $\eps$ to be specified later,
assume by contradiction that there exist a sequence $\delta_k \to 0$ and a sequence of solutions $u_k$ to the problem \eqref{fb} with right-hand-side $f_k$ such that $Lip(u_k), \|f_k\| \leq L$ and \begin{equation}\label{trap} \{x_n \leq - \delta_k\} \subset B_1 \cap \{u_k^+(x)=0\} \subset \{x_n \leq \delta_k \},\end{equation} but the $u_k$ do not satisfy the conclusion \eqref{conclusion_beta}.

Then, up to a subsequence, the $u_k$ converge uniformly on compacts to a function $u^*$. In view of \eqref{trap} and the non-degeneracy  of $u_k^+$ $2\delta_k$-away from the free boundary (Lemma \ref{deltand}), we can apply our compactness lemma and conclude that $$-L \leq \Delta u^* \leq L, \quad \text{in $B_{1/2} \cap \{x_n \neq 0\}$}$$ in the viscosity sense and also  \be\label{FBu*} ({u^*_n}^+)^2 - ({u^*_n}^-)^2= 1  \quad \hbox{on $F(u^*)=B_{1/2} \cap \{x_n=0\},$}\ee with $$u^* >0 \quad \text{in $B_{\rho_0} \cap \{x_n >0\}$}.$$
Thus, $$u^* \in C^{1,\gamma}(B_{1/2} \cap \{x_n \geq 0\}) \cap C^{1,\gamma}(B_{1/2} \cap \{x_n \leq 0\})$$ for all $\gamma$ and in view of \eqref{FBu*} we have that (for any $\bar r$ small)
$$\|u^* - (\alpha x_n^+ - \beta x_n^- )\|_{L^\infty(B_{\bar r})} \leq C(n,L) \bar r^{1+\gamma}$$
with $\alpha^2=1+\beta^2.$ If $\bar r$ is chosen depending on $\eps$ so that $$ C(n,L) \bar r^{1+\gamma} \leq \frac{\eps}{2} \bar r,$$ since the $u_k$ converge uniformly to $u^*$ on $B_{1/2}$ we obtain that for all $k$ large
$$\|u_k- (\alpha x_n^+ - \beta x_n^- )\|_{L^\infty(B_{\bar r})} \leq \eps \bar r,$$ a contradiction.
\end{proof}

In view of Lemma \ref{normalize}, after rescaling our main Theorem \ref{flatmain1} follows from the following main Theorem \ref{main_new}.

\begin{thm}\label{main_new} Let $u$ be a solution to \eqref{fb} in $B_1$ with $Lip(u) \leq L$ and $\|f\|_{L^\infty} \leq L$. There exists a universal
constant $\bar \ep>0$ such that, if \be\label{initialass}\|u - U_{\beta}\|_{L^{\infty}(B_{1})} \leq \bar \eps\quad \text{for some $0 \leq \beta \leq L,$}\ee and
\begin{equation*} \{x_n \leq - \bar \ep\} \subset B_1 \cap \{u^+(x)=0\} \subset \{x_n \leq \bar \ep \},\end{equation*} and $$\|f\|_{L^\infty(B_1)} \leq \bar \ep,$$
then $F(u)$ is $C^{1,\gamma}$ in $B_{1/2}$.

\end{thm}

Finally, we also need the following elementary lemma which holds for any arbitrary continuous function $u$.

\begin{lem}\label{elementary} Let $u$ be a continuous function. If for $\eta>0$ small, $$\|u - U_{\beta}\|_{L^{\infty}(B_{2})} \leq \eta, \quad \text{$0 \leq \beta \leq L,$}$$ and
\begin{equation*} \{x_n \leq - \eta\} \subset B_2 \cap \{u^+(x)=0\} \subset \{x_n \leq \eta \},\end{equation*} then \begin{itemize}\item If $\beta \geq \eta^{1/3}$, then $$U_\beta(x_n - \eta^{1/3}) \leq u(x) \leq U_\beta(x_n + \eta^{1/3}),\quad \text{in $B_1$}$$ \\
\item If $\beta < \eta^{1/3},$ then $$U_0(x_n - \eta^{1/3}) \leq u^+(x) \leq U_0(x_n + \eta^{1/3}),\quad \text{in $B_1$}.$$
\end{itemize}\end{lem}

\section{The linearized problem}\label{section3}

This section is devoted to the study of the linearized problem associated with  our free boundary problem \eqref{fb}, that is
 the following boundary
value problem, ($\tilde \alpha \neq 0$)
\begin{equation}\label{Neumann_p}
  \begin{cases}
    \Delta \tilde u=0 & \text{in $B_\rho \cap \{x_n \neq 0\}$}, \\
\ \\
\tilde \alpha^2  (\tilde u_n)^+ - \tilde \beta^2 (\tilde u_n)^-=0 & \text{on $B_\rho \cap \{x_n =0\}$}.
  \end{cases}\end{equation}
Here $(\tilde u_n)^{+}$ (resp. $(\tilde u_n)^-$) denotes the derivative in the $e_n$ direction of $\tilde u$ restricted to $\{x_n >0\}$ (resp. $\{x_n <0\}$). 

We remark that Theorem \ref{main_new} will follow, see Section \ref{section6}, via a compactness argument  from the regularity
properties of viscosity solutions to \eqref{Neumann_p}.

\begin{defn}A continuous function $u$ is a viscosity solution to \eqref{Neumann_p} if

(i)   $\Delta \tilde u=0  \quad \text{in $B_\rho \cap \{x_n \neq 0\}$},$ in the viscosity sense;

\

(ii) Let $\phi$ be a function of the form
$$\phi(x) = A+ px_n^+- qx_n^-+B Q(x-y)$$ with $$Q(x) = \frac 1 2 [(n-1)x_n^2 - |x'|^2], \quad y=(y',0), \quad A \in \R, B>0$$ and $$\tilde \alpha^2 p- \tilde \beta^2 q>0.$$ Then $\phi$ cannot touch $u$ strictly by below at a point $x_0= (x_0', 0) \in B_{\rho}.$

Analogously, if  $$\tilde \alpha^2 p- \tilde \beta^2 q<0$$ then $\phi$ cannot touch $u$ strictly by above at $x_0$.
\end{defn}

We wish to prove the following regularity result for viscosity solutions to the linearized problem.

\begin{thm}\label{linearreg}Let $\tilde u$ be a viscosity solution to \eqref{Neumann_p} in $B_{1/2}$ such that $\|\tilde u\|_\infty \leq 1$. There exists a universal constant $\bar C$ such that
\be\label{lr}|\tilde u(x) - \tilde u(0) -(\nabla_{x'}\tilde u(0)\cdot x' + \tilde px_n^+ - \tilde qx_n^-)| \leq \bar C r^2, \quad \text{in $B_r$}\ee for all $r \leq 1/4$ and with $\tilde \alpha^2 \tilde p -\tilde \beta^2 \tilde q=0.$
\end{thm}

In order to prove the above result, first we show in the following Theorem \ref{33}, that problem \eqref{Neumann_p} admits a classical solution.  

\begin{thm}\label{33} Let $h$ be a continuous function on $\p B_1$. There exists a (unique) classical solution $\tilde v$ to \eqref{Neumann_p} with $\tilde v = h$ on $\p B_1$, that is $\tilde v \in C^{\infty}(B_1 \cap \{x_n \geq 0\}) \cap C^{\infty}(B_1 \cap \{x_n \leq 0\}).$ In particular,
there exists a universal constant $\tilde C$ such that
\be\label{lr_again}|\tilde v(x) - \tilde v(\bar x) -(\nabla_{x'}\tilde v(\bar x)\cdot (x'- \bar x') + \tilde p(\bar x)x_n^+ - \tilde q(\bar x)x_n^-)| \leq \tilde C\|\tilde v\|_{L^\infty} r^2 , \quad \text{in $B_r(\bar x)$}\ee for all $r \leq 1/4, \bar x=(\bar x',0) \in B_{1/2}$ and with $\tilde \alpha^2 \tilde p(\bar x) -\tilde \beta^2 \tilde q(\bar x)=0.$
\end{thm}
\begin{proof}
Let $w$ be the harmonic function in $B_1 \cap \{x_n >0\}$ such that $$w = 0 \quad \text{on $B_1 \cap \{x_n=0\}$}, \quad w(x) = h(x', x_n) - h(x', -x_n) \quad \text{on $\p B_1 \cap \{x_n >0\}$.}$$ Then $w \in C^\infty(B_1\cap \{x_n \geq 0\}).$ Call $$\phi(x') = w_n(x', 0), \quad (x',0) \in B_1.$$ Let
$$\tilde v_1(x) = w(x)+ \tilde v_2(x', -x_n) \quad \text{in $\bar B_1 \cap \{x_n\geq 0\}$}$$
where $ \tilde v_2$ is the solution to the problem $$\begin{cases}
\Delta \tilde v_2 = 0 \quad \text{in $B_1 \cap \{x_n<0\}$}\\
\tilde v_2 = h \quad \text{on $\p B_1 \cap \{x_n <0\}$}\\
(\tilde v_2)_n = \tilde q \phi \quad \text{on $B_1 \cap  \{x_n=0\}$}\end{cases}$$ with $$ \tilde q = \frac{\tilde \alpha^2}{\tilde \beta^2+\tilde \alpha^2}.$$ Then it is easily verified that the function
$$\tilde v = \begin{cases} \tilde v_1  \quad \text{in $\bar B_1 \cap \{x_n \geq 0\}$}\\ \tilde v_2  \quad \text{in $\bar B_1 \cap \{x_n \leq 0\}$}\end{cases}$$
is the unique classical solution to our problem and hence it satisfies the estimate \eqref{lr_again} with $$\tilde q(\bar x) = \tilde q \phi(\bar x), \quad \tilde p(\bar x) = \tilde p \phi(\bar x), $$ and $$\tilde p = \frac{\tilde \beta^2}{\tilde \beta^2+\tilde \alpha^2}.$$
\end{proof}
Finally, to obtain our regularity result we only need to show the following fact.

\begin{thm}\label{34} Let $\tilde u$ be a viscosity solution to \eqref{Neumann_p} in $B_{1}$ such that $\|\tilde u\|_\infty \leq 1$ and let $\tilde v$ be the classical solution to \eqref{Neumann_p} in $B_{1/2}$ with boundary data $\tilde u$. Then $\tilde u = \tilde v.$
\end{thm}
\begin{proof} We prove that $\tilde v \leq \tilde u$ in $B_{1/2}$. The opposite inequality is obtained in a similar way.

Let $\ep >0, t \in \R$ and denote by
$$\tilde v_{t,\ep}(x) = \tilde v + \ep |x_n| +\ep x_n^2 -\ep -t, \quad x \in \bar B_{1/2}.$$ Since $\tilde u$ is bounded, for $t>0$ large enough
\be\label{below}\tilde v_{t, \ep}  \leq \tilde u.\ee
Let $\bar t$ be the smallest $t$ such that \eqref{below} holds and let $\bar x$ be the first touching point. We want to show that $\bar t<0$. Assume $\bar t \geq 0.$
Since
$$\tilde v_{\bar t,\ep}< \tilde u \quad \text{on $\p B_{1/2}$}$$ such touching point must belong to $B_{1/2}.$ However,
$$\Delta \tilde v_{\bar t,\ep}(x) > 0 \quad \text{in $B_{1/2} \cap \{x_n \neq 0\}$}$$
and
$$\Delta \tilde u = 0 \quad \text{in $B_{1/2} \cap \{x_n \neq 0\}$}.$$ Thus $\bar x \in B_{1/2} \cap \{x_n =0\}.$ We claim that there exists a function $\phi$ of the form
$$\phi(x) = A+ px_n^+- qx_n^-+B Q(x-y)$$ with $$Q(x) = \frac 1 2 [(n-1)x_n^2 - |x'|^2], \quad y=(y',0), \quad A \in \R, B>0$$ and $$\tilde \alpha^2 p- \tilde \beta^2 q>0,$$
such that
$\phi$ touches $ \tilde v_{\bar t,\ep}(x)$ strictly by below at $\bar x$. This would contradict the definition of viscosity solutions hence $\bar t <0$. In particular
$$ \tilde v + \ep |x_n| +\ep x_n^2 -\ep  < \tilde u \quad \text{on $B_{1/2}$}$$
and for $\ep$ going to 0 we obtain as desired
$$\tilde v \leq \tilde u \quad \text{on $B_{1/2}$}.$$We are left with the proof of the claim.
Call $$\nu' = \nabla_{x'} \tilde v(\bar x)$$ and set
$$y' = \bar x' + \frac{\nu'}{B}, \quad A = \tilde v(\bar x) - \ep - \bar t - BQ(\bar x - y)$$ with $B>0$ to be chosen later.
Then in view of the estimate \eqref{lr_again}, in order to verify that in a small neighborhood of $\bar x$
$$\phi(x) <  \tilde v_{\bar t,\ep}(x), \quad x \neq \bar x$$ we need to show that we can find $B>0,p,q$ such that  for $|x -\bar x|\neq 0$ small enough ($\tilde C$ universal)
$$\frac{B}{2}(n-1)x_n^2 -\frac{B}{2}|x'-\bar x'|^2 + px_n^+ - qx_n^- < (\tilde p + \ep)x_n^+-(\tilde q-\ep)x_n^- - \tilde C|x-\bar x|^2$$ and $$\tilde \alpha^2 p - \tilde \beta^2 q>0,$$ (for simplicity we dropped the dependence of $\tilde p, \tilde q$ on $\bar x$.)

It is then enough to choose,
$$B= 4 \tilde C, \quad p = \tilde p +\frac{\ep}{2}, \quad q=\tilde q - \frac{\ep}{2}.$$
\end{proof}

\section{Harnack inequality}\label{Harnacksec}

In this section we prove our main tool, that is a Harnack-type inequality for solutions to our free boundary problem. The results contained here will allow us to pass to the limit in the compactness argument of our improvement of flatness lemmas in Section \ref{section5}.

Throughout this section we consider a Lipschitz solution  $u$ to \eqref{fb}  with  $Lip(u) \leq L$.

We need to distinguish two cases, which we call the non-degenerate and the degenerate case.

\subsection{Non-degenerate case}

In this case our solution $u$ is trapped between two translation of a ``true" two-plane solution $U_\beta$ that is with $\beta \neq 0.$

\begin{thm}[Harnack inequality]
\label{HI}
There exists a universal constant $\bar
\ep$,  such that if $u$ satisfies at some point $x_0 \in B_2$

\be\label{osc} U_\beta(x_n+ a_0) \leq u(x) \leq U_\beta(x_n+ b_0) \quad
\text{in $B_r(x_0) \subset B_2,$}\ee
with $$ \|f\|_{L^\infty} \leq \ep^2 \beta, \quad 0 < \beta \leq L, \quad$$ and
$$b_0 - a_0 \leq \ep r, $$ for some $\ep \leq \bar \ep,$ then
$$ U_{\beta}(x_n+ a_1) \leq u(x) \leq U_\beta(x_n+ b_1) \quad \text{in
$B_{r/20}(x_0)$},$$ with
$$a_0 \leq a_1 \leq b_1 \leq b_0, \quad
b_1 -  a_1\leq (1-c)\ep r, $$ and $0<c<1$ universal.
\end{thm}

Before giving the proof we deduce an important consequence.

If $u$ satisfies \eqref{osc}
 with, say $r=1$, then we can apply Harnack inequality
repeatedly and obtain $$\label{osc2} U_\beta(x_n+ a_m) \leq u(x)
\leq U_\beta(x_n+ b_m) \quad \text{in $B_{20^{-m}}(x_0)$}, $$with
$$b_m-a_m \leq (1-c)^m\ep$$ for all $m$'s such that $$(1-c)^m 20^{m}\ep \leq \bar
\ep.$$ This implies that for all such $m$'s, the oscillation of
the function
$$\tilde u_\ep(x) = \begin{cases} \dfrac{u(x) -\alpha x_n }{\alpha\ep}  \quad \text{in $B_2^+(u) \cup F(u)$} \\ \ \\ \dfrac{u(x) -\beta x_n }{\beta\ep}  \quad \text{in $B_2^-(u)$} \end{cases}$$ in
$B_{r}(x_0), r=20^{-m}$ is less than $(1-c)^m= 20^{-\gamma m} =
r^\gamma$. Thus, the following corollary holds.

\begin{cor} \label{corollary}Let $u$ be as in Theorem $\ref{HI}$  satisfying \eqref{osc} for $r=1$. Then  in $B_1(x_0)$ $\tilde
u_\ep$ has a H\"older modulus of continuity at $x_0$, outside
the ball of radius $\ep/\bar \ep,$ i.e for all $x \in B_1(x_0)$, with $|x-x_0| \geq \ep/\bar\ep$
$$|\tilde u_\ep(x) - \tilde u_\ep (x_0)| \leq C |x-x_0|^\gamma.
$$
\end{cor}

The proof of the Harnack inequality relies on the following lemma.

\begin{lem}\label{main}There exists
a universal constant $\bar \ep>0$ such that if $u$ satisfies \begin{equation*}  u(x) \geq U_\beta(x), \quad \text{in $B_1$}\end{equation*} with \be\label{nond}\|f\|_{L^\infty(B_1)} \leq \ep^2 \beta, \quad 0 < \beta \leq L, \ee then if at $\bar x=\dfrac{1}{5}e_n$ \be\label{u-p>ep2}
u(\bar x) \geq U_\beta(\bar x_n + \ep), \ee then \be u(x) \geq
U_\beta(x_n+c\eps), \quad \text{in $\overline{B}_{1/2},$}\ee for some
$0<c<1$ universal. Analogously, if $$u(x) \leq U_\beta(x), \quad \text{in $B_1$}$$ and $$ u(\bar x) \leq U_\beta(\bar x_n - \eps),$$ then $$ u(x) \leq U_\beta(x_n - c \ep), \quad
\text{in $\overline{B}_{1/2}.$}$$
\end{lem}

\begin{proof} We prove the first statement. For notational simplicity we drop the sub-index $\beta$ from $U_\beta.$

Let \be\label{w}w=c(|x-\bar x|^{-\gamma} - (3/4)^{-\gamma})\ee be defined in the closure of the annulus
$$A:=  B_{3/4}(\bar x) \setminus \overline{B}_{1/20}(\bar x).$$
The constant $c$ is such that
$w$ satisfies the
boundary conditions
  $$\begin{cases}
    w =0 & \text{on $\p B_{3/4}(\bar x)$}, \\
    w=1 & \text{on $\p B_{1/20}(\bar x)$}.
  \end{cases} $$
 Then, for a fixed $\gamma > n-2,$

 $$\Delta w  \geq k(\gamma, n)=k(n) > 0, \quad 0 \leq w \leq 1  \quad \text{on $A$.}$$
Extend $w$ to be equal
to 1 on $B_{1/20}(\bar x).$

Notice that since $x_n >0$ in $B_{1/10}(\bar x)$  and $u \geq U$ in $B_1$ we get
\begin{equation*} B_{1/10}(\bar x) \subset B_1^+(u). \end{equation*}

Thus $u-U \geq
0$ and solves $\Delta (u-U) =f$  in $B_{1/10}(\bar x)$ and we can apply Harnack inequality
to obtain
\be\label{HInew} u(x)
- U(x) \geq c(u(\bar x)- U(\bar x)) - C \|f\|_{L^\infty} \quad \text{in $\overline B_{1/20}(\bar x)$}. \ee
From the assumptions \eqref{nond} and  \eqref{u-p>ep2}  we conclude that (for $\ep$ small enough)
\be\label{u-p>cep} u
- U \geq \alpha c\ep - C \alpha \ep^2 \geq \alpha c_0\ep \quad \text{in $\overline B_{1/20}(\bar x)$}. \ee
Now set $\psi =1-w$ and

\begin{equation*} v(x)= U(x_n - \ep c_0 \psi(x)), \quad x \in \overline B_{3/4}(\bar x),\end{equation*} and for $t
\geq 0,$
$$v_t(x)= U(x_n - \ep c_0 \psi(x)+t\ep), \quad x \in \overline B_{3/4}(\bar x).
$$

Then,

$$  v_0(x)=U(x_n - \ep c_0 \psi(x)) \leq
U(x) \leq u(x) \quad x \in \overline B_{3/4}(\bar x).$$

Let $\bar t$ be the largest $t \geq 0$ such that
$$v_{t}(x) \leq u(x) \quad \text{in $\overline B_{3/4}(\bar x)$}.$$

We want to show that $\bar t \geq c_0.$ Then we get the desired statement. Indeed,
$$u(x) \geq v_{\bar t}(x) = U(x_n - \ep c_0 \psi + \bar t \ep) \geq U(x_n + c\ep) \quad \text{in $B_{1/2}  \subset \subset B_{3/4}(\bar x)$}$$ with $c$ universal. In the last inequality we used that $\|\psi\|_{L^\infty(B_{1/2})} <1.$

Suppose $\bar t < c_0$. Then at some
$\tilde x \in \overline B_{3/4}(\bar x)$ we have $$v_{\bar t}(\tilde x) =
u(\tilde x).$$ We show that such touching point can only occur on $\overline B_{1/20}(\bar x).$
Indeed, since $w\equiv 0$ on $\p B_{3/4}(\bar x)$ from the definition of $v_t$ we get that for $\bar t < c_0$

$$v_{\bar t}(x) =U(x_n -\ep c_0\psi(x) + \bar t \ep) < U(x) \leq u(x)\quad \textrm{on  $\p B_{3/4}(\bar x)$}.$$

We now show that $\tilde x$ cannot belong to the annulus $A$.
Indeed, $$ \Delta v_{\bar t}\geq \beta \ep c_0 k(n) > \ep^2 \beta \geq \|f\|_{\infty}, \quad \textrm{in
$A^+(v_{\bar t}) \cup A^-(v_{\bar t})$}$$ for $\ep$ small enough.

Also,
$$(v_{\bar t}^+)_\nu^2 - (v_{\bar t}^-)_\nu^2 = 1 + \ep^2 c_0^2 |\nabla \psi|^2  - 2\ep c_0 \psi_n \quad \text{on $F(v_{\bar t}) \cap A$}.$$

Thus,

$$(v_{\bar t}^+)_\nu^2 - (v_{\bar t}^-)_\nu^2 > 1  \quad \text{on $F(v_{\bar t}) \cap A$}$$
as long as
$$\psi_n <0  \quad \text{on $F(v_{\bar t}) \cap A$}.$$
This can be easily verified from the formula for $\psi$ (for $\ep$ small enough.)

 Thus,
$v_{\bar t}$ is a strict subsolution to $\eqref{fb}$ in $A$ which lies below $u$, hence by the definition of viscosity solution,
$\tilde x$ cannot belong to $A.$

Therefore, $\tilde x \in \overline B_{1/20}(\bar x)$ and
$$u(\tilde x)=v_{\bar t}(\tilde x)  = U(\tilde x_n + \bar t \ep) \leq U(\tilde x) + \alpha \bar t \eps < U(\tilde x) + \alpha c_0 \ep$$ contradicting \eqref{u-p>cep}.

The proof of the second statement follows from a similar argument.
\end{proof}

We can now prove our Theorem \ref{HI}.

\vspace{3mm}

\textit{Proof of Theorem $\ref{HI}$.} Assume without loss of generality that $x_0=0, r=1.$ We distinguish three cases.

\vspace{1mm}

{\it Case 1.} $a_0 < - 1/5.$ In this case it follows from \eqref{osc} that $B_{1/10} \subset \{u<0\}$ and
$$0 \leq v(x):= \frac{u(x) - \beta(x_n+a_0)}{\beta \eps} \leq 1$$
with $$|\Delta v| \leq \eps \quad \text{in $B_{1/10}$.}$$ The desired claim follows from standard Harnack inequality applied to the function $v$.

\vspace{1mm}

{\it Case 2.} $a_0 > 1/5.$ In this case it follows from \eqref{osc} that $B_{1/5} \subset \{u>0\}$ and
$$0 \leq v(x):= \frac{u(x) - \alpha(x_n+a_0)}{\alpha \eps} \leq 1$$
with $$|\Delta v| \leq \eps \quad \text{in $B_{1/5}$.}$$ Again, the desired claim follows from standard Harnack inequality for $v$.

\vspace{1mm}

{\it Case 3.} $|a_0| \leq 1/5.$ Assumption \eqref{osc} gives that
$$U_\beta(x_n + a_0) \leq u(x) \leq U_\beta(x_n+a_0+\eps) \quad \text{in $B_1$.}$$ Assume that  (the other case is treated similarly)
\be\label{top} u(\bar x) \geq U_{\beta}(\bar x_n + a_0 + \frac{\eps}{2}), \quad \bar x=\frac{1}{5}e_n.\ee

Call $$v(x) := u(x-a_0 e_n), \quad x \in B_{4/5}.$$ Then the inequality above reads
$$U_\beta(x_n) \leq v(x) \leq U_\beta(x_n+\eps) \quad \text{in $B_{4/5}$.}$$From \eqref{top} $$v(\bar x) \geq U_\beta(\bar x_n + \frac{\eps}{2}).$$ Then, by Lemma \ref{main}, $$v(x) \geq U_\beta(x_n + c\eps), \quad \text{in $B_{2/5}$}$$ which gives the desired improvement
$$u(x) \geq U_{\beta}(x+a_0+ c\eps) \quad \text{in $B_{3/5}.$}$$
\qed

\subsection{Degenerate case}

In this case, the negative part of $u$ is negligible and the positive part is close to a one-plane solution (i.e. $\beta=0$).

\begin{thm}[Harnack inequality]\label{HIdg}There exists a universal constant $\bar
\ep$,  such that if $u$ satisfies at some point $x_0 \in B_2$

\be\label{oscdg} U_0(x_n+ a_0) \leq u^+(x) \leq U_0(x_n+ b_0) \quad
\text{in $B_r(x_0) \subset B_2,$}\ee
with  $$\|u^-\|_{L^\infty} \leq \ep^2, \quad \|f\|_{L^\infty} \leq \ep^4,$$ and
$$b_0 - a_0 \leq \ep r, $$ for some $\ep \leq \bar \ep,$ then
$$ U_{0}(x_n+ a_1) \leq u^+(x) \leq U_0(x_n+ b_1) \quad \text{in
$B_{r/20}(x_0)$},$$ with
$$a_0 \leq a_1 \leq b_1 \leq b_0, \quad
b_1 -  a_1\leq (1-c)\ep r, $$ and $0<c<1$ universal.
\end{thm}

We can argue as in the nondegenerate case and get the following result.

\begin{cor} \label{corollary4}Let $u$ be as in Theorem \ref{HI}  satisfying \eqref{oscdg} for $r=1$. Then  in $B_1(x_0)$ $$\tilde
u_\ep:= \frac{u^+(x) - x_n}{\ep}$$ has a H\"older modulus of continuity at $x_0$, outside
the ball of radius $\ep/\bar \ep,$ i.e for all $x \in B_1(x_0)$, with $|x-x_0| \geq \ep/\bar\ep$
$$|\tilde u_\ep(x) - \tilde u_\ep (x_0)| \leq C |x-x_0|^\gamma.
$$
\end{cor}

The proof of the Harnack inequality can be deduced from  the following lemma, as in the one-phase case \cite{D}.

\begin{lem}\label{main2}There exists
a universal constant $\bar \ep>0$ such that if $u$ satisfies
\begin{equation*}  u^+(x) \geq U_0(x), \quad \text{in $B_1$}
\end{equation*}
with \be\label{dg} \|u^-\|_{L^\infty} \leq \ep^2, \quad \|f\|_{L^\infty} \leq \ep^4, \ee then if at $\bar x=\dfrac{1}{5}e_n$ \be\label{u-p>ep2d}
u^+(\bar x) \geq U_0(\bar x_n + \ep), \ee then \be u^+(x) \geq
U_0(x_n+c\eps), \quad \text{in $\overline{B}_{1/2},$}\ee for some
$0<c<1$ universal. Analogously, if $$u^+(x) \leq U_0(x), \quad \text{in $B_1$}$$ and $$ u^+(\bar x) \leq U_0(\bar x_n - \eps),$$ then $$ u^+(x) \leq U_0(x_n - c \ep), \quad
\text{in $\overline{B}_{1/2}.$}$$
\end{lem}

\begin{proof} We prove the first statement. The proof follows the same line as in the nondegenerate case.

Since $x_n >0$ in $B_{1/10}(\bar x)$  and $u^+ \geq U_0$ in $B_1$ we get
\begin{equation*} B_{1/10}(\bar x) \subset B_1^+(u). \end{equation*}

Thus $u-x_n \geq
0$ and solves $\Delta (u-x_n) =f$  in $B_{1/10}(\bar x)$ and we can apply Harnack inequality and the assumptions \eqref{dg} and  \eqref{u-p>ep2d}
to obtain that (for $\ep$ small enough)
\be\label{u-p>cep2} u
- x_n \geq  c_0\ep \quad \text{in $\overline B_{1/20}(\bar x)$}. \ee
Let $w$ be as in the proof of Lemma \ref{main} and  $\psi =1-w$. Set
\begin{equation*} v(x)= (x_n - \ep c_0 \psi(x))^+ - \ep^2 C_1 (x_n - \ep c_0 \psi(x))^-, \quad x \in \overline B_{3/4}(\bar x),\end{equation*} and for $t
\geq 0,$
$$v_t(x)= (x_n - \ep c_0 \psi +  t \ep)^+ -   \ep^2 C_1(x_n - \ep c_0 \psi(x) + t \ep)^-, \quad x \in \overline B_{3/4}(\bar x).
$$
Here $C_1$ is a universal constant to be made precise later.
We claim that
$$  v_0(x)=v(x) \leq u(x) \quad x \in \overline B_{3/4}(\bar x).$$

This is readily verified in the set where $u$ is non-negative using that $u \geq x_n^+.$  To prove our claim in the set where $u$ is negative we wish to use the following fact:
\be\label{negu}u^- \leq C x_n^- \ep^2, \quad \text{in $B_{\frac{19}{20}}$, $C$ universal}.\ee This estimate is easily obtained using that $\{u<0\} \subset \{x_n <0\},$ $\|u^-\|_\infty < \ep^2$ and the comparison principle with the function $w$ satisfying $$\Delta w = -\eps^4 \quad \text{in $B_1 \cap \{x_n <0\}$}, \quad w=u^- \quad \text{on $\p (B_1 \cap \{x_n <0\})$.}$$
Thus our claim immediately follows from the fact that for $x_n<0$ and $C_1 \geq C,$
$$\ep^2 C_1(x_n - \ep c_0 \psi(x)) \leq Cx_n \ep^2.$$

Let $\bar t$ be the largest $t \geq 0$ such that
$$v_{t}(x) \leq u(x) \quad \text{in $\overline B_{3/4}(\bar x)$}.$$

We want to show that $\bar t \geq c_0.$ Then we get the desired statement. Indeed, it is easy to check that if
$$u(x) \geq v_{\bar t}(x) = (x_n - \ep c_0 \psi + \bar t \ep)^+ -   \ep^2 C_1 (x_n - \ep c_0 \psi(x) + \bar t \ep)^-\quad \text{in $ B_{3/4}(\bar x)$}$$ then
$$u^+(x) \geq U_0(x_n + c \ep) \quad \text{in $ B_{1/2} \subset \subset B_{3/4}(\bar x)$}$$
with $c$ universal, $c < c_0 \inf_{B_1/2} w$

Suppose $\bar t < c_0$. Then at some
$\tilde x \in \overline B_{3/4}(\bar x)$ we have $$v_{\bar t}(\tilde x) =
u(\tilde x).$$ We show that such touching point can only occur on $\overline B_{1/20}(\bar x).$
Indeed, since $w\equiv 0$ on $\p B_{3/4}(\bar x)$ from the definition of $v_t$ we get that for $\bar t < c_0$
$$v_{\bar t}(x) =(x_n -\ep c_0 + \bar t \ep)^+ - \ep^2 C_1 (x_n -\ep c_0 + \bar t \ep)^-< u(x)\quad \textrm{on  $\p B_{3/4}(\bar x)$}.$$

In the set where $u \geq 0$ this can be seen using that $u \geq x_n^+$ while in the set where $u<0$ again we can use the estimate \eqref{negu}.

We now show that $\tilde x$ cannot belong to the annulus $A$.
Indeed, $$\Delta v_{\bar t}\geq \ep^3c_0 k(n) > \ep^4 \geq \|f\|_{\infty}, \quad \textrm{in
$A^+(v_{\bar t}) \cup A^-(v_{\bar t})$}$$ for $\ep$ small enough.

Also,
$$(v_{\bar t}^+)_\nu^2 - (v_{\bar t}^-)_\nu^2 = (1 -\ep^4 C_1^2)(1 + \ep^2 c_0^2 |\nabla \psi|^2  - 2\ep c_0 \psi_n) \quad \text{on $F(v_{\bar t}) \cap A$}.$$

Thus,
$$(v_{\bar t}^+)_\nu^2 - (v_{\bar t}^-)_\nu^2 > 1  \quad \text{on $F(v_{\bar t}) \cap A$}$$
as long as $\ep$ is small enough (as in the non-degenerate case one can check that $\inf_ {F(v_{\bar t}) \cap A} (-\psi_n)>c > 0$, $c$ universal.)
 Thus,
$v_{\bar t}$ is a strict subsolution to $\eqref{fb}$ in $A$ which lies below $u$, hence by definition $\tilde x$ cannot belong to $A.$

Therefore, $\tilde x \in \overline B_{1/20}(\bar x)$ and
$$u(\tilde x)=v_{\bar t}(\tilde x)  = (\tilde x_n + \bar t \ep)< \tilde x_n +  c_0 \ep$$ contradicting \eqref{u-p>cep2}.
\end{proof}

\section{Improvement of flatness}\label{section5}

 In this section we prove our key ``improvement of flatness" lemmas. As in Section \ref{Harnacksec}, we need to distinguish two cases.

\subsection{Non-degenerate case} In this case our solution $u$ is trapped between two translations of a  two-plane solution $U_\beta$ with $\beta \neq 0.$ We plan to show that when we restrict to smaller balls, $u$ is trapped between closer translations of another two-plane solution (in a different system of coordinates).


\begin{lem}[Improvement of flatness] \label{improv1}Let $u$  satisfy
\begin{equation}\label{flat_1}U_\beta(x_n -\ep) \leq u(x) \leq U_\beta(x_n +
\ep) \quad \text{in $B_1,$} \quad 0\in F(u),
\end{equation} with $0 <  \beta \leq L$ and
$$\|f\|_{L^\infty(B_1)} \leq \ep^2\beta.$$

If $0<r \leq r_0$ for $r_0$ universal, and $0<\ep \leq \ep_0$ for some $\ep_0$
depending on $r$, then

\begin{equation}\label{improvedflat_2_new}U_{\beta'}(x \cdot \nu_1 -r\frac{\ep}{2})  \leq u(x) \leq
U_{\beta'}(x \cdot \nu_1 +r\frac{\ep }{2}) \quad \text{in $B_r,$}
\end{equation} with $|\nu_1|=1,$ $ |\nu_1 - e_n| \leq \tilde C\ep$ , and $|\beta -\beta'| \leq \tilde C\beta \ep$ for a
universal constant $\tilde C.$

\end{lem}

\begin{proof}We divide the proof of this Lemma into 3 steps.

 \vspace{2mm}

\textbf{Step 1 -- Compactness.} Fix $r \leq r_0$ with $r_0$ universal (the precise $r_0$ will be given in Step 3). Assume by contradiction that we
can find a sequence $\ep_k \rightarrow 0$ and a sequence $u_k$ of
solutions to \eqref{fb} in $B_1$ with  right hand side $f_k$ with $L^\infty$ norm bounded by $\ep_k^2\beta_k$, such that
\begin{equation}\label{flat_k}U_{\beta_k}(x_n -\ep_k) \leq u_k(x) \leq U_{\beta_k}(x_n +
\ep_k) \quad \text{for $x \in B_1$,  $0 \in F(u_k),$}
\end{equation} with $L \geq \beta_k > 0$,
but $u_k$ does not satisfy the conclusion \eqref{improvedflat_2_new} of the lemma.

Set ($\alpha_k^2=1+\beta_k^2$),$$ \tilde{u}_{k}(x)= \begin{cases}\dfrac{u_k(x) - \alpha_k x_n}{\alpha_k \ep_k}, \quad x \in
B_1^+(u_k) \cup F(u_k) \\ \ \\ \dfrac{u_k(x) - \beta_k x_n}{\beta_k\ep_k}, \quad x \in
B_1^-(u_k).\end{cases}
$$
Then \eqref{flat_k} gives,
\begin{equation}\label{flat_tilde**} -1 \leq \tilde{u}_{k}(x) \leq 1
\quad \text{for $x \in B_1$}.
\end{equation}

From Corollary \ref{corollary}, it follows that the function
$\tilde u_{k}$ satisfies \be\label{HC}|\tilde u_{k}(x) - \tilde
u_{k} (y)| \leq C |x-y|^\gamma,\ee for $C$ universal and
$$|x-y| \geq \ep_k/\bar\ep, \quad x,y \in B_{1/2}.$$ From \eqref{flat_k} it clearly follows that
$F(u_k)$ converges to $B_1 \cap \{x_n=0\}$ in the Hausdorff
distance. This fact and \eqref{HC} together with Ascoli-Arzela
give that as $\ep_k \rightarrow 0$ the graphs of the
$\tilde{u}_{k}$ converge (up to a
subsequence) in the Hausdorff distance to the graph of a H\"older
continuous function $\tilde{u}$ over $B_{1/2}$. Also, up to a subsequence
$$\beta_k \to \tilde \beta \geq 0$$ and hence $$\alpha_k \to \tilde \alpha = \sqrt{1+\tilde \beta^2}.$$

\vspace{2mm}

\textbf{Step 2 -- Limiting Solution.} We now show that $\tilde u$
solves the following linearized problem (transmission problem)
\begin{equation}\label{Neumann}
  \begin{cases}
    \Delta \tilde u=0 & \text{in $B_{1/2} \cap \{x_n \neq 0\}$}, \\
\ \\
\tilde \alpha^2 (\tilde u_n)^+ - \tilde\beta^2 (\tilde u_n)^-=0 & \text{on $B_{1/2} \cap \{x_n =0\}$}.
  \end{cases}\end{equation}

  Since $$|\Delta u_{k}| \leq \ep_k^2\beta_k \quad
\text{in $B_1^+(u_k) \cup B^-_1(u_k)$},$$ one easily deduces that $\tilde
u$ is harmonic in $B_{1/2} \cap \{x_n  \neq 0\}$.

Next, we prove that $\tilde u$ satisfies the boundary
condition in \eqref{Neumann} in the viscosity sense.

Let $\tilde \phi$ be a function of the form
$$\tilde \phi(x) = A+ px_n^+- qx_n^- + B Q(x-y)$$ with $$Q(x) = \frac 1 2 [(n-1)x_n^2 - |x'|^2], \quad y=(y',0), \quad A \in \R, B >0$$ and $$\tilde \alpha^2 p- \tilde \beta^2 q>0.$$ Then we must show that $\tilde \phi$ cannot touch $u$ strictly by below at a point $x_0= (x_0', 0) \in B_{1/2}$ (the analogous statement by above follows with a similar argument.)

Suppose that such a $\tilde \phi$ exists and let $x_0$ be the touching point.

Let \be\label{biggamma}\Gamma(x) = \frac{1}{n-2} [(|x'|^2 + |x_n-1|^2)^{\frac{2-n}{2}} - 1]\ee and let

\be\label{biggammak}\Gamma_{k}(x)=  \frac{1}{B\ep_k}\Gamma(B\ep_k(x-y)+ A B \ep_k^2 e_n).\ee

Now,  call
$$\phi_k(x)= a_k \Gamma^+_k(x) - b_k \Gamma^-_k(x) + \alpha_k (d_k^+(x))^2\ep_k^{3/2} +\beta_k(d_k^-(x))^2\ep_k^{3/2}$$
where
$$a_k=\alpha_k(1+\ep_k p), \quad b_k=\beta_k(1+\ep_k q)$$ and $d_k(x)$ is the signed distance from $x$ to $\p B_{\frac{1}{B\ep_k}}(y+e_n(\frac{1}{B\ep_k}-A\ep_k)).$

Finally, let

$$ \tilde{\phi}_{k}(x)= \begin{cases}\dfrac{\phi_k(x) - \alpha_k x_n}{\alpha_k \ep_k}, \quad x \in
B_1^+(\phi_k) \cup F(\phi_k) \\ \ \\ \dfrac{\phi_k(x) - \beta_k x_n}{\beta_k\ep_k}, \quad x \in
B_1^-(\phi_k).\end{cases}
$$

By Taylor's theorem
$$ \Gamma(x)= x_n + Q(x) + O(|x|^3) \quad x \in
B_1,$$
thus it is easy to verify that
$$ \Gamma_k(x)= A\ep_k + x_n + B\ep_kQ(x-y) + O(\ep_k^2) \quad x \in
B_1,$$
with the constant in $O(\ep_k^2)$ depending on $A,B,$ and $|y|$ (later this constant will depend also on $p,q.$).

It follows that in $
B_1^+(\phi_k) \cup F(\phi_k) $ ($Q^y(x)=Q(x-y)$)
 $$ \tilde{\phi}_{k}(x)= A+BQ^y + p x_n + A\ep_k p + Bp\ep_kQ^y + \ep_k^{1/2}d_k^2+ O(\ep_k)
$$
and analogously in  $
B_1^-(\phi_k)$

$$ \tilde{\phi}_{k}(x)= A+BQ^y + q x_n + A\ep_k p + Bq\ep_kQ^y + \ep_k^{1/2}d^2_k + O(\ep_k).
$$

Hence, $\tilde \phi_k$ converges uniformly to $\tilde \phi$ on $B_{1/2}$. Since $\tilde u_k$ converges uniformly to $\tilde u$ and $\tilde \phi$ touches $\tilde u$ strictly by below at $x_0$, we conclude that there exist a sequence of constants $c_k \to 0$ and of points $x_k \to x_0$ such that the function
$$\psi_k(x) = \phi_k(x+\ep_k c_k e_n)$$
touches $u_k$ by below at $x_k$. We thus get a contradiction if we prove that $\psi_k$ is a strict subsolution to our free boundary problem, that is \begin{equation*}\label{fbpsi*} \left \{
\begin{array}{ll}
   \Delta \psi_k > \ep_k^2\beta_k \geq \|f_k\|_{\infty},   & \hbox{in $B_1^+(\psi_k) \cup B_1^-(\psi_k),$}\\
\ \\
(\psi_k^+)_\nu^2 - (\psi_k^-)^2_\nu >1, & \hbox{on $F(\psi_k)$.}\\
\end{array}\right.
\end{equation*}

It is easily  checked that away from the free boundary
$$\Delta \psi_k \geq \beta_k \ep_k^{3/2} \Delta d_k^2(x+\ep_kc_k e_n)$$
and the first condition is satisfied for $k$ large enough.

Finally, since on the zero level set $|\nabla \Gamma_k|=1$ and $|\nabla d^2_k|=0$ the free boundary condition reduces to showing that
$$a_k^2-b_k^2 >1.$$
Using the definition of $a_k, b_k$ we need to check that
$$(\alpha_k^2p^2-\beta_k^2q^2)\ep_k + 2(\alpha_k^2 p -\beta_k^2 q)  >0. $$
This inequality holds for $k$ large in view of the fact that $$\tilde \alpha^2 p -\tilde \beta^2 q >0. $$

Thus $\tilde u$ is a solution to the linearized problem.

\vspace{2mm}

\textbf{Step 3 -- Contradiction.} According to estimate \eqref{lr}, since $\tilde u(0)=0$ we obtain that
$$|\tilde u - (x' \cdot \nu' + \tilde px_n^+ -\tilde qx_n^-)| \leq C r^2, \quad x\in B_r,$$
with
$$\tilde \alpha ^2 \tilde p -\tilde \beta^2 \tilde q=0, \quad |\nu'| = |\nabla_{x'} \tilde u (0)| \leq C.$$

Thus, since $\tilde u_k$ converges uniformly to $\tilde u$ (by slightly enlarging $C$) we get that
\be\label{ukest}|\tilde u_k - (x' \cdot \nu' + \tilde px_n^+ -\tilde qx_n^-)| \leq C r^2, \quad x\in B_r.\ee

Now set,
$$\beta'_k =  \beta_k(1+\ep_k \tilde q), \quad \nu_k = \frac{1}{\sqrt {1+ \ep_k^2 |\nu'|^2}}(e_n + \ep_k (\nu', 0)).$$

Then,
$$\alpha'_k = \sqrt{1+{\beta'_k}^2} =  \alpha_k(1+\ep_k \tilde p) + O(\ep_k^2), \quad \nu_k =e_n + \ep_k (\nu', 0) + \ep_k^2 \tau, \quad |\tau|\leq C,$$
where to obtain the first equality we used that $\tilde \alpha ^2 \tilde p -\tilde \beta^2 \tilde q=0$ and hence $$\frac {\beta_k^2}{\alpha_k^2} \tilde q= \tilde p + O(\ep_k).$$

With these choices we can now show that (for $k$ large and $r \leq r_0$)
$$\widetilde{U}_{\beta'_k}(x\cdot \nu_k -\ep_k\frac r 2) \leq \tilde u_k(x) \leq \widetilde{U}_{\beta'_k}(x\cdot \nu_k +\ep_k\frac r 2), \quad \text{in $B_r$}$$ where again we are using the notation:
$$ \widetilde{U}_{\beta'_k}(x)= \begin{cases}\dfrac{ \widetilde{U}_{\beta'_k}(x) - \alpha_k x_n}{\alpha_k \ep_k}, \quad x \in
B_1^+( \widetilde{U}_{\beta'_k}) \cup F( \widetilde{U}_{\beta'_k}) \\ \ \\ \dfrac{ \widetilde{U}_{\beta'_k}(x) - \beta_k x_n}{\beta_k\ep_k}, \quad x \in
B_1^-( \widetilde{U}_{\beta'_k}).\end{cases}
$$

This will clearly imply that
$$U_{\beta'_k}(x\cdot \nu_k -\ep_k\frac r 2) \leq u_k(x) \leq U_{\beta'_k}(x\cdot \nu_k +\ep_k\frac r 2), \quad \text{in $B_r$}$$
and hence leads to a contradiction.

In view of \eqref{ukest} we need to show that in $B_r$
$$\widetilde{U}_{\beta'_k}(x\cdot \nu_k -\ep_k\frac r 2) \leq (x' \cdot \nu' + \tilde px_n^+ -\tilde qx_n^-) - Cr^2$$

and
$$ \widetilde{U}_{\beta'_k}(x\cdot \nu_k +\ep_k\frac r 2) \geq (x' \cdot \nu' + \tilde px_n^+ -\tilde qx_n^-) + Cr^2.$$

Let us show the second inequality.
In the set where \be\label{<}x\cdot \nu_k +\ep_k\frac r 2 <0\ee by definition we have that
$$ \widetilde{U}_{\beta'_k}(x\cdot \nu_k +\ep_k\frac r 2) = \frac{1}{\beta_k \ep_k}( \beta'_k(x\cdot \nu_k +\ep_k\frac r 2 ) -  \beta_k x_n)$$ which from the formula for $\beta_k', \nu_k$ gives
$$ \widetilde{U}_{\beta'_k}(x\cdot \nu_k +\ep_k\frac r 2)  \geq   x' \cdot \nu' + \tilde qx_n +\frac r 2 - C_0 \ep_k.$$
Using \eqref{<} we then obtain
$$ \widetilde{U}_{\beta'_k}(x\cdot \nu_k +\ep_k\frac r 2)  \geq   x' \cdot \nu' + \tilde p x_n^+ - \tilde qx_n^- +\frac r 2 - C_1 \ep_k.$$ Thus to obtain the desired bound it  suffices to fix $r_0 \leq 1/(4C)$ and take $k$ large enough.

The other case can be argued similarly.
\end{proof}

\subsection{Degenerate case} In this case, the negative part of $u$ is negligible and the positive part is close to a one-plane solution (i.e. $\beta=0$). We prove below that in this setting only $u^+$ enjoys an improvement of flatness.

\begin{lem}[Improvement of flatness] \label{improv4_deg}Let $u$  satisfy
\begin{equation}\label{flat***}U_0(x_n -\ep) \leq u^+(x) \leq U_0(x_n +
\ep) \quad \text{in $B_1,$} \quad 0\in F(u),
\end{equation} with $$\|f\|_{L^\infty(B_1)} \leq \ep^4,$$ and $$ \|u^-\|_{L^\infty(B_1)} \leq \ep^2.$$

If $0<r \leq r_1$ for $r_1$ universal, and $0<\ep \leq \ep_1$ for some $\ep_1$
depending on $r$, then

\begin{equation}\label{improvedflat_2}U_0(x \cdot \nu_1 -r\frac{\ep}{2})  \leq u^+(x) \leq
U_0(x \cdot \nu_1 +r\frac{\ep }{2}) \quad \text{in $ B_r,$}
\end{equation} with $|\nu_1|=1,$ $ |\nu_1 - e_n| \leq C\ep$  for a
universal constant $C.$

\end{lem}

\begin{proof}We argue similarly as in the non-degenerate case.

\vspace{2mm}

\textbf{Step 1 -- Compactness.} Fix $r \leq r_1$ with $r_1$ universal (made precise in Step 3). Assume by contradiction that we
can find a sequence $\ep_k \rightarrow 0$ and a sequence $u_k$ of
solutions to \eqref{fb} in $B_1$ with  right hand side $f_k$ with $L^\infty$ norm bounded by $\ep_k^4$, such that
\begin{equation}\label{flat_k4}U_{0}(x_n -\ep_k) \leq u^+_k(x) \leq U_{0}(x_n +
\ep_k) \quad \text{for $x \in B_1$,  $0 \in F(u_k),$}
\end{equation} with $$\|u_k^-\|_\infty \leq \ep^2_k$$
but $u_k$ does not satisfy the conclusion \eqref{improvedflat_2} of the lemma.

Set $$ \tilde{u}_{k}(x)= \dfrac{u_k(x) -  x_n}{ \ep_k}, \quad x \in
B_1^+(u_k) \cup F(u_k) $$
Then \eqref{flat_k4} gives,
\begin{equation}\label{flat_tilde*} -1 \leq \tilde{u}_{k}(x) \leq 1
\quad \text{for $x \in
B_1^+(u_k) \cup F(u_k) $}.
\end{equation}

As in the non-degenerate case, it follows from Corollary \ref{corollary4} that as $\ep_k \rightarrow 0$ the graphs of the
$\tilde{u}_{k}$ converge (up to a
subsequence) in the Hausdorff distance to the graph of a H\"older
continuous function $\tilde{u}$ over $B_{1/2} \cap \{x_n \geq 0\}$.

\vspace{2mm}

\textbf{Step 2 -- Limiting Solution.} We now show that $\tilde u$ solves the following Neumann problem
\begin{equation}\label{Neumann4}
  \begin{cases}
    \Delta \tilde u=0 & \text{in $B_{1/2} \cap \{x_n > 0\}$}, \\
\ \\
\tilde u_n =0 & \text{on $B_{1/2} \cap \{x_n =0\}$}.
  \end{cases}\end{equation}

 As before, the interior condition follows easily thus we focus on the boundary condition.

Let $\tilde \phi$ be a function of the form
$$\tilde \phi(x) = A+ px_n + B Q(x-y)$$ with $$Q(x) = \frac 1 2 [(n-1)x_n^2 - |x'|^2], \quad y=(y',0), \quad A \in \R, B>0$$ and $$p >0.$$ Then we must show that $\tilde \phi$ cannot touch $u$ strictly by below at a point $x_0= (x_0', 0) \in B_{1/2}$.
Suppose that such a $\tilde \phi$ exists and let $x_0$ be the touching point.

Let  $\Gamma_k$ be as in the proof of the non-degenerate case (see \eqref{biggammak}). Call

$$\phi_k(x)= a_k \Gamma^+_k(x)+ (d_k^+(x))^2\ep_k^2, \quad a_k=(1+\ep_k p) $$
where
$d_k(x)$ is the signed distance from $x$ to $\p B_{\frac{1}{B\ep_k}}(y+e_n(\frac{1}{B\ep_k}-A\ep_k)).$

Let
$$ \tilde{\phi}_{k}(x)=\dfrac{\phi_k(x) -  x_n}{ \ep_k}. $$

As in the previous case, it follows that in $
B_1^+(\phi_k) \cup F(\phi_k) $ ($Q^y(x)=Q(x-y)$)
$$ \tilde{\phi}_{k}(x)= A+BQ^y + p x_n + A\ep_k p + Bp\ep_kQ^y + \ep_kd_k^2+ O(\ep_k).
$$

Hence, $\tilde \phi_k$ converges uniformly to $\tilde \phi$ on $B_{1/2} \cap \{x_n \geq 0\}$. Since $\tilde u_k$ converges uniformly to $\tilde u$ and $\tilde \phi$ touches $\tilde u$ strictly by below at $x_0$, we conclude that there exist a sequence of constants $c_k \to 0$ and of points $x_k \to x_0$ such that the function
$$\psi_k(x) = \phi_k(x+\ep_k c_k e_n)$$
touches $u_k$ by below at $x_k \in B_1^+(u_k) \cup F(u_k)$. We claim that $x_k$ cannot belong to $B_1^+(u_k)$. Otherwise, in a small neighborhood $N$ of $x_k$  we would have that
$$ \Delta \psi_k > \ep_k^4 \geq \|f_k\|_{\infty} = \Delta u_k, \quad \text{$\psi_k < u_k$ in $N\setminus\{x_k\}, \psi_k(x_k) = u_k(x_k)$} $$
a contradiction.

Thus $x_k \in F(u_k) \cap \p B_{\frac{1}{B\ep_k}}(y+e_n(\frac{1}{B\ep_k}-A\ep_k-\eps_kc_k)).$ For simplicity we call $$\mathcal B: =B_{\frac{1}{B\ep_k}}(y+e_n(\frac{1}{B\ep_k}-A\ep_k-\eps_kc_k)).$$ Let $N_\rho$ be a small neighborhood of $x_k$ of size $\rho$. Since $$\|u^-_k\|_\infty \leq \eps^2_k, \quad u_k^+ \geq (x_n-\eps_k)^+$$ as in the proof of Harnack inequality using the fact that  $x_k \in F(u_k) \cap \p \mathcal B$ we can conclude by the comparison principle that

$$u^-_k  \leq c \eps^2_k (d(x, \p \mathcal B))^-, \quad \text{in $N_{\frac{3}{4}\rho}$}$$
where $d$ denotes again the signed distance from $x$ to $\p \mathcal B.$

Let \begin{equation}\label{Psi}\Psi_k(x) = \begin{cases}\psi_k & \text{in $\mathcal B$}\\  \ & \ \\ c\eps^2_k (3d(x,\p \mathcal B) + d^2(x,\p \mathcal B)) & \text{outside of $\mathcal B.$}\end{cases}\ee

Then  $\Psi_k$ touches $u_k$ strictly by below at $x_k \in F(u_k) \cap F(\Psi_k)$.

We reach a contradiction if we show that
$$(\Psi_k^+)_\nu^2 - (\Psi_k^-)^2_\nu >1, \quad \hbox{on $F(\Psi_k)$.}$$

This is equivalent to showing that
$$a_k^2 - c \eps_k^4 >1$$ or
$$(1+\ep_k p)^2 -c \ep_k^4> 1.$$

This holds for $k$ large enough, since $p>0.$ We finally reached a contradiction.

\

\textbf{Step 3 -- Contradiction.} In this step we can argue as in the final step of the proof of Lemma 4.1 in [D]. \end{proof}

\section{Proof of the main Theorems}\label{section6}

In this section we exhibit the proofs of our main results, Theorem \ref{flatmain1} and Theorem \ref{Lipmain}. As already pointed out, Theorem \ref{Lipmain} will follow via a blow-up analysis
from the flatness result. Thus, first we present the proof of Theorem \ref{flatmain1} based on the improvement of flatness lemmas of the previous section.

\subsection{Proof of Theorem \ref{flatmain1}.}
To complete the analysis of the degenerate case, we need to deal with the situation when $u$ is close to a one-plane solution and however the size of $u^-$ is not negligible. Precisely, we prove the following lemma.

\begin{lem}\label{finalcase}
Let $u$  solve \eqref{fb} in $B_2$ with $$\|f\|_{L^\infty(B_1)} \leq \ep^4$$ and satisfy
\begin{equation}\label{flat**}U_0(x_n -\ep) \leq u^+(x) \leq U_0(x_n +
\ep) \quad \text{in $B_1,$} \quad 0\in F(u),
\end{equation}$$\|u^-\|_{L^{\infty}(B_2)} \leq \bar C\eps^2,  \quad \|u^-\|_{L^\infty(B_1)} > \ep^2,$$ for a universal constant $\bar C.$ If $\ep \leq \ep_2$ universal,  then the rescaling
$$u_\ep(x) = \ep^{-1/2}u(\ep^{1/2}x)$$ satisfies in $B_1$
$$U_{\beta'}(x_n - C'\ep^{1/2}) \leq u_{\ep}(x) \leq U_{\beta'}(x_n + C'\ep^{1/2})$$ with $\beta' \sim \eps^2$ and $C'>0$ depending on $\bar C$.
\end{lem}
\begin{proof} For notational simplicity we set $$v = \frac{u^-}{\ep^2}.$$
From our assumptions we can deduce that $$F(v) \subset \{-\ep \leq x_n \leq \ep\}$$ \be\label{neg}v \geq 0 \quad \text{in $B_2 \cap \{x_n \leq -\ep\}$}, \quad v \equiv 0 \quad \text{in  $B_2 \cap \{x_n > \ep\}.$}\ee Also,
$$|\Delta v| \leq \ep^2, \quad \text{in $B_2 \cap \{x_n < -\ep\}$},$$ and \be\label{C}0 \leq v \leq \bar C \quad \text{on $\p B_2$,}\ee \be\label{barx}v(\bar x) >1 \quad \text{at some point $\bar x$ in $B_1.$}\ee

Thus, using comparison with the function $w$ such that $$\Delta w = -\ep^2 \quad \text{in $D:=B_2 \cap \{x_n < \eps\}$ and $w= v$ on $\p D$}$$ we obtain that for some $k>0$ universal
\be\label{ke}v \leq k|x_n -\ep|, \quad \text{in $B_1$}.\ee  This fact forces the point $\bar x$ in \eqref{barx} to belong to $B_1\cap \{x_n < -\ep\}$ at a fixed distance $\delta$ from $x_n = -\ep.$

Now, let $w$ be the harmonic function in $B_1 \cap \{x_n < -\ep\}$ such that $$w=0 \quad \text{on $B_1 \cap \{x_n=-\ep\}$}, \quad w=v \quad \text{on $\p B_1 \cap  \{x_n \leq  -\ep\}$}.$$ By the maximum principle we conclude that
$$ w+ \ep^2(|x|^2-3) \leq v \quad \text{on $B_1 \cap \{x_n < -\ep\}$.}$$
Also, for $\ep$ small, in view of \eqref{ke} we obtain that
$$  w- k\ep(|x|^2-3) \geq v \quad \text{on $\p (B_1 \cap \{x_n < -\ep\})$}$$ and hence also in the interior. Thus we conclude that
\be\label{w-v} |w-v| \leq c\eps \quad \text{in $B_1 \cap \{x_n < -\ep\}$}. \ee In particular this is true at $\bar x$ which forces \be\label{wbarx} w(\bar x) \geq 1/2.\ee By expanding $w$ around $(0, -\ep)$ we then obtain, say in $B_{1/2} \cap  \{x_n \leq -\ep\}$
$$|w - a |x_n + \ep|| \leq C |x|^2 +C\ep.$$
This combined with \eqref{w-v} gives that
$$ |v - a|x_n+\ep|| \leq C\ep, \quad \text{in $B_{\ep^{1/2}} \cap  \{x_n \leq -\ep\}$.}$$
Moreover, in view of \eqref{wbarx} and the fact that $\bar x$ occurs at a fixed distance from $\{x_n = -\ep\}$ we deduce from Hopf lemma that
$$a \geq c>0$$ with $c$ universal. In conclusion (see \eqref{ke})
$$\label{u-}  |u^- - b\ep^2|x_n+\ep|| \leq C\ep^3, \quad \text{in $B_{\ep^{1/2}} \cap  \{x_n \leq -\ep\}$,} \quad u^- \leq b\ep^2 |x_n-\ep|, \quad \text{in $B_1$}$$ with $b$ comparable to a universal constant.

Combining the two inequalities above and the assumption \eqref{flat**} we conclude that in $B_{\ep^{1/2}}$
$$ (x_n - \ep)^+ - b\ep^2(x_n-C\ep)^- \leq u(x) \leq (x_n+\ep)^+ - b \ep^2 (x_n+C\ep)^-$$
with $C>0$ universal and $b$ larger than a universal constant. Rescaling, we obtain that in $B_1$
$$ (x_n - \ep^{1/2})^+ - \beta'(x_n-C\ep^{1/2})^- \leq u_\ep(x) \leq (x_n+\ep^{1/2})^+ - \beta'(x_n+C\ep^{1/2})^-$$
with $\beta' \sim  \ep^2$. We finally need to check that this implies the desired conclusion in $B_1$
$$ \alpha'(x_n - C\ep^{1/2})^+ - \beta'(x_n-C\ep^{1/2})^- \leq u_\ep(x) \leq \alpha'(x_n+C\ep^{1/2})^+ - \beta' (x_n+C\ep^{1/2})^-$$ with $\alpha'^2=1+\beta'^2 \sim 1+ \ep^4.$ This clearly holds in $B_1$ for $\ep$ small, say  by possibly enlarging $C$ so that $C \geq 2.$
\end{proof}

We are finally ready to exhibit the proof of Theorem \ref{main_new}, which as already observed, immediately gives the result of Theorem \ref{flatmain1}.

\

\textit{Proof of Theorem $\ref{main_new}.$} Let us fix $\bar r >0$ to be a universal constant such that
$$\bar r \leq r_0, r_1, 1/16,$$
with $r_0,r_1$ the universal constants in the improvement of flatness Lemmas \ref{improv1}-\ref{improv4_deg}. Also, let us fix a universal constant $\tilde \ep>0$ such that
$$\tilde \ep \leq \ep_0(\bar r), \frac{\ep_1(\bar r)}{2}, \frac{1}{2\tilde C}, \frac{\ep_2}{2}$$ with $\ep_0,\ep_1,\ep_2, \tilde C,$ the constants in the Lemmas \ref{improv1}-\ref{improv4_deg}-\ref{finalcase} .
Now, let $$\bar \ep = \tilde \ep^3.$$ We distinguish two cases. For notational simplicity we assume that $u$ satisfies our assumptions in the ball $B_2$ and $0 \in F(u)$.

\vspace{2mm}

\textit{Case 1.} $\beta \geq \tilde \ep.$

\vspace{2mm}

In this case, in view of Lemma \ref{elementary} and our choice of $\tilde \ep$, we obtain that $u$ satisfies the assumptions of Lemma \ref{improv1}, \begin{equation}\label{flat_1*}U_\beta(x_n -\tilde \ep) \leq u(x) \leq U_\beta(x_n +
\tilde \ep) \quad \text{in $B_1,$} \quad 0\in F(u),
\end{equation} with $0 <  \beta \leq L$ and
$$\|f\|_{L^\infty(B_1)} \leq \tilde \ep^3 \leq \tilde \ep^2\beta.$$ Thus we can conclude that, ($\beta_1=\beta'$)
$$U_{\beta_1}(x \cdot \nu_1 -\bar r\frac{\tilde \ep}{2})  \leq u(x) \leq
U_{\beta_1}(x \cdot \nu_1 +\bar r\frac{\tilde \ep }{2}) \quad \text{in $B_{\bar r},$}
$$ with $|\nu_1|=1,$ $ |\nu_1 - e_n| \leq \tilde C\tilde \ep$ , and $|\beta -\beta_1| \leq \tilde C\beta \tilde \ep$. In particular, by our choice of $\tilde \ep$ we have
$$\beta_1 \geq \tilde \ep/2.$$We can therefore rescale and iterate the argument above. Precisely, set ($k=0,1,2....$)
$$\rho_k = \bar r^k, \quad \ep_k = 2^{-k}\tilde \ep$$
and $$u_k(x) = \frac{1}{\rho_k} u(\rho_k x), \quad f_k(x) = \rho_k f(\rho_k x).$$
Also, let $\beta_k$ be the constants generates at each $k$-iteration, hence satisfying ($\beta_0=\beta$)
$$|\beta_k-\beta_{k+1}| \leq \tilde C \beta_k \ep_k.$$
Then  we obtain by induction that each $u_k$ satisfies
\begin{equation}U_{\beta_k}(x \cdot \nu_k -\ep_k)  \leq u_k(x) \leq
U_{\beta_k}(x \cdot \nu_k +\ep_k ) \quad \text{in $B_1,$}
\end{equation} with $|\nu_k|=1,$ $ |\nu_k - \nu_{k+1}| \leq \tilde C\tilde \ep_k$ ($\nu_0=e_n$.)

\vspace{2mm}

\textit{Case 2.}  $\beta < \tilde \ep.$

\vspace{2mm}

In view of Lemma \ref{elementary} we conclude that
$$U_0(x_n -\tilde \ep) \leq u^+(x) \leq U_0(x_n +
\tilde \ep) \quad \text{in $B_1.$}
$$ Moreover, from the assumption \eqref{initialass} and the fact that $\beta < \tilde \ep$ we also obtain that
$$\|u^-\|_{L^\infty (B_1)} < 2\tilde \ep.$$
Call $(\ep')^2 = 2\tilde \ep.$ Then $u$ satisfies the assumptions of the (degenerate) improvement of flatness Lemma \ref{improv4_deg}.
$$U_0(x_n -\ep') \leq u^+(x) \leq U_0(x_n +
\ep') \quad \text{in $B_1,$}
$$ with $$\|f\|_{L^{\infty}(B_1)} \leq (\ep')^4, \quad \|u^-\|_{L^\infty(B_1)} < \ep'^2.$$
We conclude that
$$U_0(x \cdot \nu_1 -\bar r\frac{\ep'}{2})  \leq u^+(x) \leq
U_0(x \cdot \nu_1 +\bar r\frac{\ep' }{2}) \quad \text{in $ B_{\bar r},$}
$$ with $|\nu_1|=1,$ $ |\nu_1 - e_n| \leq C\ep'$  for a
universal constant $C.$
We now rescale as in the previous case and set ($k=0,1,2....$)
$$\rho_k = \bar r^k, \quad \ep_k = 2^{-k}\ep'$$
and $$u_k(x) = \frac{1}{\rho_k} u(\rho_k x), \quad f_k(x) = \rho_k f(\rho_k x).$$
We can iterate our argument and obtain that (with $|\nu_k|=1,$ $ |\nu_k - \nu_{k+1}| \leq C\ep_k$)
\begin{equation}\label{deltoid} U_0(x \cdot \nu_k -\ep_k)  \leq u_k^+(x) \leq
U_0(x \cdot \nu_k +\ep_k) \quad \text{in $ B_{1},$}
\end{equation}
as long as we can verify that
$$\|u_k^-\|_{L^\infty(B_1)} < \ep^2_k.$$
Let $\bar k$ be the first integer $\bar k> 1$ for which this fails, that is
$$\|u^-_{\bar k}\|_{L^\infty(B_1)} \geq \ep^2_{\bar k},$$ and $$\|u_{\bar k-1}^-\|_{L^\infty(B_1)} < \ep^2_{\bar k-1}.$$
Also, $$U_0(x \cdot \nu_{\bar k-1} -\ep_{\bar k-1})  \leq u_{\bar k-1}^+(x) \leq
U_0(x \cdot \nu_{\bar k-1} +\ep_{\bar k_-1}) \quad \text{in $ B_{1}.$}
$$
As argued several times (see for example \eqref{negu}), we can then conclude from the comparison principle that
$$u^-_{\bar k-1} \leq M|x_n-\ep_{\bar k-1}|\ep^2_{\bar k-1}\quad \text{in $ B_{19/20},$}
$$ for a universal constant $M>0. $ Thus, by rescaling we get that
$$\|u^-_{\bar k}\|_{L^\infty(B_2)} < \bar C \ep_{\bar k}^2$$ with $\bar C$ universal (depending on the fixed $\bar r$). We obtain that $u_{\bar k}$ satisfies all the assumptions of Lemma \ref{finalcase} and hence  the rescaling
$$v(x) = \ep_{\bar k}^{-1/2}u_{\bar k}(\ep_{\bar k}^{1/2}x)$$ satisfies in $B_1$
$$U_{\beta'}(x_n - C'\ep_{\bar k}^{1/2}) \leq v(x) \leq U_{\beta'}(x_n + C'\ep_{\bar k}^{1/2})$$ with $\beta' \sim \eps_{\bar k}^2.$
Call $\eta=\bar C \ep_{\bar k}^{1/2}.$
Then $v$ satisfies our free boundary problem in $B_1$ with right hand side $$g(x)= \ep_{\bar k}^{1/2} f_{\bar k}(\ep_{\bar k}^{1/2}x)$$ and the flatness assumption
$$U_{\beta'}(x_n -\eta) \leq v(x) \leq U_{\beta'}(x_n + \eta)$$
Since $\beta' \sim \ep^2_{\bar k}$ with a universal constant, $$\|g\|_{L^\infty(B_1)} \leq \ep_{\bar k}^{1/2} \ep_{\bar k}^{4} \leq \eta^2 \beta'$$
as long as $\tilde \ep \leq C''$ depending on $\bar C$. In conclusion choosing $\tilde \eps \leq \frac{\eps_0(\bar r)^4}{2 \bar C^4}$, $v$ falls under the assumptions of the (non-degenerate) improvement of flatness Lemma \ref{improv1} and we can use an iteration argument as in Case 1.
\qed

\

  \subsection{Proof of Theorem \ref{Lipmain}.} Although not strictly necessary, we use the following Liouville type result for global viscosity solutions to a two-phase homogeneous free boundary problem, that could be of independent interest.

\begin{lem}\label{Liouville} Let $U$ be a global viscosity solution to \begin{equation}\label{fbglobal} \left \{
\begin{array}{ll}
   \Delta U = 0,   & \hbox{in $\{U>0\}\cup \{U \leq 0\}^0,$}\\
\ \\
(U_\nu^+)^2 - (U_\nu^-)^2= 1, & \hbox{on $F(U):= \partial \{U>0\}.$} \\
\end{array}\right.
\end{equation}
Assume that $F(U) = \{x_n =g(x'), x' \in \R^{n-1}\}$ with $Lip(g) \leq M$.  Then $g$ is linear and $U(x)=U_\beta(x)$ for some $\beta \geq 0.$
\end{lem}
\begin{proof}Assume for simplicity, $0 \in F(U)$. Also, balls (of radius $\rho$ and centered at $0$) in $\R^{n-1}$ are denoted by $\mathcal B_\rho.$

By the regularity theory in \cite{C1} , since $U$ is a solution in $B_2$, the free boundary $F(U)$ is $C^{1,\gamma}$ in $B_1$ with a bound depending only on $n$ and on $M$. Thus,
$$|g(x') - g(0) -\nabla g(0) \cdot x'| \leq C |x'|^{1+\alpha}, \quad x' \in \mathcal B_1 $$ with $C$ depending only on $n, M.$
Moreover, since $U$ is a global solution, the rescaling
$$g_R(x') = \frac{1}{R} g(Rx'), \quad x' \in \mathcal B_2$$
which preserves the same Lipschitz constant as $g$, satisfies the same inequality as above i.e.
$$|g_R(x') - g_R(0) -\nabla g_R(0) \cdot x'| \leq C |x'|^{1+\alpha}, \quad x' \in \mathcal B_1.$$
This reads,
$$|g(Rx') - g(0) -\nabla g(0) \cdot Rx'| \leq C R |x'|^{1+\alpha}, \quad x' \in \mathcal B_1.$$
Thus,
$$|g(y') - g(0) -\nabla g(0) \cdot y'| \leq C \frac{1}{R^\alpha}|y'|^{1+\alpha}, \quad y' \in \mathcal B_R.$$

Passing to the limit as $R \to \infty$ we obtain the desired claim.
\end{proof}

\

\textit{Proof of Theorem \ref{Lipmain}.} Let $\bar \eps$ be the universal constant in Theorem \ref{main_new}. Consider the blow-up sequence $$u_k(x) = \frac{u(\delta_k)}{\delta_k}$$ with $\delta_k \to 0$ as $k \to \infty$. Each $u_k$ solves \eqref{fb} with right hand side $$f_k(x) = \delta_k f(\delta_k x)$$ and $$\|f_k (x)\| \leq  \delta_k \|f\|_{L^\infty} \leq \bar \eps$$ for $k$ large enough. Standard arguments (see for example \cite{ACF}) using the uniform Lischitz continuity of the $u_k$'s and the nondegeneracy of their positive part $u_k^+$ (see  Lemma \ref{deltand}) imply that (up to a subsequence)
$$u_k \to \tilde u \quad \text{uniformly on compacts}$$
and
$$\{u_k^+=0\} \to \{\tilde u = 0\} \quad \text{in the Hausdorff distance}.$$

The blow-up limit $\tilde u$ solves the global homogeneous two-phase free boundary problem
\begin{equation}\label{fbglobal*} \left \{
\begin{array}{ll}
    \Delta \tilde u = 0,   & \hbox{in $\{\tilde u>0\} \cup \{\tilde u \leq 0\}^0$} \\
\ \\
 (\tilde u_\nu^+)^2 - (\tilde u_\nu^-)^2= 1, & \hbox{on $F(\tilde u):= \p \{\tilde u>0\}.$}  \\
\end{array}\right.
\end{equation}

Since $F(u)$ is a Lipschitz graph in  a neighborhood of 0, it follows from Lemma \ref{Liouville} that $\tilde u$ is a two-plane solutions, $\tilde u= U_\beta$ for some $\beta \geq 0$. Thus, for $k$ large enough
$$\|u_k - U_\beta\|_{L^\infty} \leq \bar \ep$$
and
\begin{equation*} \{x_n \leq - \bar \ep\} \subset B_1 \cap \{u^+_k(x)=0\} \subset \{x_n \leq \bar \ep \}.\end{equation*}  Therefore, we can apply our flatness Theorem \ref{main_new} and conclude that $F(u_k)$ and hence $F(u)$ is smooth.

\subsection{Flatness and $\ep$-monotonicity}\label{subsection7}

The flatness results which are present in the literature (see, for instance \cite{C2}), are often stated in terms of ``$\ep$- monotonicity" along a large cone of directions $\Gamma(\theta_0, e)$ of axis $e$ and opening $\theta_0$. Precisely, a function $u$ is said to be $\ep$-monotone ($\ep>0$ small) along the direction $\tau$ in the cone $\Gamma(\theta_0, e)$ if for every $\ep' \geq \ep$,
$$u(x+\ep'\tau) \leq u(x).$$

A variant of Theorem \ref{flatmain1} states the following.

\begin{thm}\label{DDFFSS} Let $u$ be a solution to \eqref{fb} in $B_1$, $0 \in F(u).$ Suppose that $u^+$ is non-degenerate. Then there exist $\theta_0 < \pi/2$ and $\ep_0>0$ such that if $u^+$ is $\ep$-monotone along every direction in $\Gamma(\theta_0, e_n)$ for some $\ep \leq \ep_0$, then $u^+$ is fully monotone in $B_{1/2}$ along any direction in $\Gamma(\theta_1,e_n)$ for some $\theta_1$ depending on $\theta_0, \ep_0.$ In particular $F(u)$ is the graph of a Lipschitz function.
\end{thm}

Geometrically, the $\ep$-monotonicity of $u^+$ can be interpreted as $\ep$-closeness of $F(u)$ to the graph of a Lipschitz function. Our flatness assumption
requires $\ep$-closeness of $F(u)$ to a hyperplane. While this looks like a somewhat stronger assumption, it is indeed a natural one since it is satisfied for example by rescaling of solutions around a ``regular" point of the free boundary. Moreover, if $\|f\|_\infty$ is small enough, depending on $\ep$, it is not hard to check that  $\ep$-flatness of $F(u)$ implies $c\ep$-monotonicity of $u^+$ along the directions of a flat cone, for a $c$ depending on its opening.

The proof of Theorem \ref{DDFFSS} follows immediately from the following elementary lemma:

\begin{lem} Let $u$ be a solution to \eqref{fb} in $B_1$, $0 \in F(u).$ Suppose that $u^+$ is Lipschitz and non-degenerate. Assume that  $u^+$ is $\ep$-monotone along every direction in $\Gamma(\theta_0, e_n)$ for some $\ep \leq \ep_0$, then there exist a radius $r_0>0$ and  $\delta_0>0$ depending on $\ep_0, \theta_0$ such that $u^+$ is $\delta_0$-flat in $B_{ r_0}$, that is \begin{equation*} \{x_n \leq - \delta_0\} \subset B_{r_0} \cap \{u^+(x)=0\} \subset \{x_n \leq \delta_0 \}.\end{equation*}  \end{lem}

\section{More general operators and free boundary conditions}\label{section8}

\subsection{The set up}

In this section we analyze the free boundary problem (\ref{fbcv}), that is
\begin{equation}\label{fbcvfinal}
\left \{
\begin{array}{ll}
\mathcal{L}u = f, & \hbox{in $\Omega^+(u) \cup
\Omega^-(u),$} \\
\  &  \\
u_\nu^+=G(u^-_\nu,x), & \hbox{on $F(u):= \partial \Omega^+(u) \cap \Omega,$}
\\
&
\end{array}
\right.
\end{equation}
where $f$ is continuous {in $\Omega^+(u) \cup \Omega^-(u)$} with $%
\|f\|_{L^\infty(\Omega)}\leq L,$  and 
\begin{equation*}
\mathcal{L}=\sum_{i,j=1}^na_{ij}(x)D_{ij} + {\bf b} \cdot \nabla, \quad a_{ij}\in C^{0, \bar{\gamma}}(\Omega), {\bf b}\in C(\Omega)\cap
L^\infty(\Omega),
\end{equation*}
is uniformly elliptic with constants $0<\lambda\leq\Lambda$.

We recall that our assumptions on $G$ are:
\smallskip

\begin{itemize}
\item[(H1)]
$G(\eta,\cdot)\in C^{0,\bar\gamma}(\Omega)$ uniformly in $\eta;\quad G(\cdot,x)\in C^{1,\bar\gamma}([0,L])$ for every $x\in \Omega.$

\item[(H2)]
$G'(\cdot, x)>0$ with $G(0,x)\geq\gamma_0>0$ uniformly in $x$.
\item[(H3)]
There exists $N>0$ such that $\eta^{-N}G(\eta,x)$ is strictly decreasing in $\eta$, uniformly in $x$.
\end{itemize}
\smallskip

We assume that $0\in F(u)$ and that $a_{ij}(0)=\delta_{ij}$. Also, for notational convenience we set
$$
G_0(\beta)=G(\beta,0).
$$

\smallskip

Let $U_\beta$ be the two-plane solution to \eqref{fbcvfinal} when $\mathcal{L}%
=\Delta,f \equiv 0 $ and $G=G_0$, i.e.
\begin{equation*}
U_\beta(x) = \alpha x_n^+ - \beta x_n^-, \quad \beta \geq 0, \quad \alpha
=G_0(\beta).
\end{equation*}

The following definitions parallel those in Section \ref{section2}.

\begin{defn}
\label{defsubcv} We say that $v \in C(\Omega)$ is a $C^2$ strict
(comparison) subsolution (resp. supersolution) to (\ref{fbcvfinal}) in $\Omega$,
if  $v\in C^2(\overline{\Omega^+(v)}) \cap C^2(\overline{\Omega^-(v)})$ and the following conditions are satisfied:

\begin{enumerate}
\item $\mathcal{L}v> f $ (resp. $< f $) in $%
\Omega^+(v) \cup \Omega^-(v)$;

\item If $x_0 \in F(v)$, then
\begin{equation*}
v_\nu^+(x_0)>G(v_\nu^-(x_0),x_0) \quad (\text{resp. $v_\nu^+(x_0)<G(v_%
\nu^-(x_0),x_0), \:\:v_\nu^+(x_0) \neq 0$.)}
\end{equation*}
\end{enumerate}
\end{defn}

\smallskip

Observe that the free boundary of a strict comparison sub/supersolution is $%
C^2$.

\begin{defn}
\label{defnhsolcv} Let $u$ be a continuous function in $\Omega$. We say that
$u$ is a viscosity solution to (\ref{fbcv}) in $\Omega$, if the following
conditions are satisfied:

\begin{enumerate}
\item $\mathcal{L}u = f$ in $\Omega^+(u) \cup
\Omega^-(u)$ in the viscosity sense;

\item Any (strict) comparison subsolution $v$ (resp. supersolution) cannot touch
$u$ by below (resp. by above) at a point $x_0 \in F(v)$ (resp. $F(u)$.)
\end{enumerate}
\end{defn}

From here after, most of the statements and proofs parallel those in Sections 2 to 6. Thus,
we only point out the main differences as much as possible.

\subsection{Compactness and localization}

As for problem (\ref{fb}), we prove some basic lemmas to reduce the statement
of the flatness theorem to a proper normalized situation. We start with the compactness Lemma
\ref{compact_delta} which generalizes to operators of the form
\begin{equation*}
\mathcal L_*^k=\sum a^k_{ij}D_{ij}
\end{equation*}
with $a_{ij}^k \in C^{0,\bar \gamma}$ uniformly elliptic with constants $\lambda, \Lambda$ and free boundary conditions given by a $G_k$
satisfying the hypotheses (H1)-(H3).

\smallskip

\begin{lem}
\label{compvar} Let $u_k$ be a sequence of (Lipschitz) viscosity solutions to
\begin{equation}  \label{fbcv2}
\left \{
\begin{array}{ll}
|\mathcal L_*^k  u_k|\leq M, & \hbox{in
$\Omega^+(u_k) \cup \Omega^-(u_k),$} \\
\  &  \\
(u_k^+)_\nu=G_k((u^-_k)_\nu,x), & \hbox{on $F(u_k).$}%
\end{array}%
\right.
\end{equation}
\smallskip

Assume that:

\begin{enumerate}
\item $a^k_{ij}\to a_{ij}, u_k \to u^*$
uniformly on compact sets,

\item $G_k(\eta,\cdot)\to G(\eta,\cdot)$ on compact sets, uniformly on $0\leq \eta\leq L=Lip(u_k),$

\item $\{u_k^+=0\} \to \{(u^*)^+=0\}$ in the Hausdorff distance.
\end{enumerate}

Then
\begin{equation*}
|\sum a_{ij}D_{ij}u^*|\leq M, \quad \text{in $%
\Omega^+(u^*) \cup \Omega^-(u^*)$}
\end{equation*}
and $u^*$ satisfies the free boundary condition
\begin{equation*}
(u^*)^+_\nu=G((u^*)^-_\nu,x)\quad \hbox{on $F(u^*)$,}
\end{equation*}
both in the viscosity sense.
\end{lem}

\begin{proof} Call $$\mathcal {L_*}:=\sum a_{ij}D_{ij}.$$ The proof that \begin{equation*}
|\mathcal L_* u^*|\leq M, \quad \text{in $%
\Omega^+(u^*) \cup \Omega^-(u^*)$}
\end{equation*} is standard. We show for example that 
\begin{equation*}
\mathcal{L_*}u^*+M\geq 0,\quad \text{in $%
\Omega^+(u^*)$.}
\end{equation*}

Let $v \in C^2(\Omega^+(u^*))$ touch $u^*$ by above at $\bar{x} \in\Omega^+(u^*)$ and assume by contradiction that 
\begin{equation*}
\mathcal{L}_*v(\bar x) +M<0.\end{equation*}
Without loss of generality we can assume that $v$ touches $u^*$ strictly by above (otherwise we replace $v$ with $v+\frac{\eta}{2n\Lambda}|x-\bar{x}|^2$ and $\eta$ small.)
Then, since $u_k\to u^*$ uniformly in
compact sets and $\{u_k^+=0\} \to \{(u^*)^+=0\}$ in the Hausdorff distance, there exists $x_k\to \bar{x}$ and constants $c_k\to 0$ such
that $v+c_k$ touches by above $u_k$ at $x_k\in\Omega^+(u_k)$, for $k$ large.
Then, since $|\mathcal L^k_*u_k (x_k)| \leq M$ we must have
\begin{equation*}
\mathcal{L}^k_* v
(x_k)+M\geq 0.
\end{equation*}
This implies, for $k\to\infty$,
\begin{equation*}
\mathcal{L}_* v(\bar{x}%
)+M\geq 0
\end{equation*}
which is a contradiction.

We now prove that the free boundary condition holds. Let $v$ be a strict comparison super solution such that, \begin{equation}\label{L*}
\mathcal{L}_*v +M<0, \quad \text{in $\Omega^+(v) \cup \Omega^-(v),$}\end{equation}
and 
$$v^+_\nu < G(v^-_\nu, x), \quad v^+_\nu(x) \neq 0 \quad \text{on $F(v)$.}$$ Assume $v$ touches $u^*$ strictly by above at a point $\bar x \in F(u^*) \cap F(v)$
and for notational simplicity let $\nu(%
\bar{x})= e_n$. Also, we can assume that the free boundaries $F(v)$ and $F(u^*)$ touch strictly and that \eqref{L*} holds up to $F(v)$. Otherwise, say $v^+_n(\bar x) > 0$, we replace $v$ with $\bar v(x) = v(x+ \eta |x'-\bar x'|^2 e_n) + \eta |dist(x, F(v))| - dist(x,F(v))^2$,  ($\eta$ small). Then, for a suitable $c_k\to 0$,  $v(x+c_k e_n)$ touches
by above $u_k$ at $x_k$ with $x_k \to \bar x$. Then, either for every (large) $k$ we have $%
x_k\in\Omega^+(u_k) \cup \Omega^-(u_k)$ or there exists a subsequence, that
we still call ${x_k}$, such that $x_k\in F(u_k)$ for every large $k$.

In the first case, we have
\begin{equation*}
\sum a^k_{ij}(x_k)D_{ij} v (x_k+c_k e_n)+M\geq 0.
\end{equation*}
while in the second case,
\begin{equation*}
\bar{v}_{\nu_k}^+(x_k+c_k e_n) \geq G_k(v_{\nu_k}^-(x_k+c_k e_n ),x_k%
)
\end{equation*}
and we easily reach a contradiction for $k$ large.\end{proof}

Lemma \ref{deltand} on the non-degeneracy of the positive part $\delta$-away
from the free boundary continues to hold unaltered; only choose $$w(x)= \frac{G_0(0)}{2\gamma}(1-|x|^{-\gamma}).$$
 The analogue of Lemma
\ref{normalize} is the following:

\begin{lem}
\label{loc1var} Let $u$ be a Lipschitz solution to \eqref{fbcv} in $B_1$,
with $Lip(u) \leq L$, $\|b\|_\infty, \|f\|_\infty \leq L$. For any $\varepsilon >0$ there exist $\bar \delta,
\bar r >0$  such that if
\begin{equation*}
\{x_n \leq - \delta\} \subset B_1 \cap \{u^+(x)=0\} \subset \{x_n \leq
\delta \},
\end{equation*}
with $0 \leq \delta \leq \bar \delta,$ then
\begin{equation}  \label{conclusion_beta1}
\|u - U_{\beta}\|_{L^{\infty}(B_{\bar r})} \leq \varepsilon \bar r
\end{equation}
for some $0 \leq \beta \leq L.$
\end{lem}

\begin{proof}
Given $\varepsilon>0$ and $\bar r$ depending on $\varepsilon$ to be
specified later, assume by contradiction that there exist a sequence $%
\delta_k \to 0$ and a sequence of solutions $u_k$ to the
problem \eqref{fbcv2} with $M=L+L^2$, such that $Lip(u_k)\leq L$ and
\begin{equation}  \label{trapp}
\{x_n \leq - \delta_k\} \subset B_1 \cap \{u_k^+(x)=0\} \subset \{x_n \leq
\delta_k \},
\end{equation}
but the $u_k$ do not satisfy the conclusion \eqref{conclusion_beta1}.

Then, up to a subsequence, the $u_k$ converge uniformly on compact set to a
function $u^*$. In view of \eqref{trapp} and the non-degeneracy of $u_k^+,$ $
\delta_k$-away from the free boundary (see remark above), we can apply
our compactness Lemma \ref{compvar} and conclude that, for some
$\tilde{\mathcal{L}}:= \sum \tilde a_{ij} D_{ij}$ and $\tilde G$ in our class,
\begin{equation*}
|\tilde{\mathcal{L}} u^*|\leq M, \quad \text{in $%
B_{1/2} \cap \{x_n \neq 0\}$}
\end{equation*}
and
\begin{equation}  \label{FB1u*}
(u^*)_n^+= \tilde G((u^*)_n^-,x) \quad \hbox{on $F(u^*)=B_{1/2} \cap \{x_n=0\},$}
\end{equation}
in the viscosity sense, with
\begin{equation*}
u^* >0 \quad \text{in $B_{\rho_0} \cap \{x_n >0\}$}.
\end{equation*}

Thus, by $L^p$ Schauder estimates
\begin{equation*}
u^* \in C^{1,\tilde{\gamma}}(B_{1/2} \cap \{x_n \geq 0\}) \cap C^{1,\tilde{%
\gamma}}(B_{1/2} \cap \{x_n \leq 0\})
\end{equation*}
for all $\tilde{\gamma}<1$ and (for any
$\bar r$ small)
\begin{equation*}
\|u^* - (\alpha x_n^+ - \beta x_n^- )\|_{L^\infty(B_{\bar r})} \leq C(n,L)
\bar r^{1+\tilde\gamma}
\end{equation*}
with $\beta=(u^*)^-_n (0)$ and $\alpha=(u^*)^+_n(0)>0$. Thus, from \eqref{FB1u*}%
, we have $\alpha=\tilde G_0(\beta)$.

Then we reach a contradiction as in Lemma \ref{normalize}.
\end{proof}

In view of the lemma above, after proper rescaling, Theorem \ref{flatmain2}
follows from the following result.

\begin{thm}
\label{main_newvar} Let $u$ be a Lipschitz solution to \eqref{fbcv} in $B_1,$
with $Lip(u) \leq L$. There exists a universal constant $\bar \varepsilon>0$
such that, if
\begin{equation}  \label{initialas}
\|u - U_{\beta}\|_{L^{\infty}(B_{1})} \leq \bar \varepsilon\quad \text{for
some $0 \leq \beta \leq L,$}
\end{equation}
\begin{equation*}
\{x_n \leq - \bar \varepsilon\} \subset B_1 \cap \{u^+(x)=0\} \subset \{x_n
\leq \bar \varepsilon \},
\end{equation*}
and
\begin{equation*}
[a_{ij}]_{C^{0,\bar\gamma}(B_1)} \leq \bar \varepsilon,\quad \|\mathbf{b}%
\|_{L^\infty(B_1)} \leq \bar \varepsilon, \quad \|f\|_{L^\infty(B_1)} \leq
\bar \varepsilon, \end{equation*} $$[G(\eta,\cdot)]_{C^{0,\bar \gamma}(B_1)} \leq \bar \ep
, \quad \forall 0 \leq \eta \leq L,$$
then $F(u)$ is $C^{1,\gamma}$ in $B_{1/2}$.
\end{thm}

\subsection{Linearized problem}

The linearized problem becomes, ($\tilde \alpha > 0$)
\begin{equation}  \label{Neumann_p_var}
\begin{cases}
\Delta \tilde u=0 & \text{in $B_\rho \cap \{x_n \neq 0\}$}, \\
\  &  \\
\tilde \alpha (\tilde u)_n^+ -\tilde \beta G^{\prime}_0( \tilde \beta) (\tilde
u)_n^-=0 & \text{on $B_\rho \cap \{x_n =0\}$},%
\end{cases}%
\end{equation}
with $\tilde{\alpha}=G_0(\tilde{\beta}).$

Setting $\zeta^2=\tilde{\alpha}$ and $\xi^2=\tilde \beta G^{\prime}_0( \tilde \beta)$ we can write the free boundary condition as
\begin{equation*}
\zeta^2 \tilde u_n^+-\xi^2 \tilde u_n^-=0.
\end{equation*}

As a consequence, all the Definitions and conclusions in Section \ref{section3} hold, in particular
Theorems \ref{linearreg}, \ref{33} and \ref{34}.


\section{The non-degenerate case for general free boundary problems.}\label{section9}

In this section, we recover the improvement of flatness lemma in the non-degenerate case, that is when the solution is trapped between parallel two-plane solutions $U_\beta$ at $\varepsilon$
distance, with $\beta>0.$ First we need the Harnack inequality.

\subsection{Harnack inequality.}

As in Section \ref{Harnacksec}, Harnack inequality follows from the
following basic lemma.

\begin{lem}
\label{mainvar} Let $u$ be a viscosity solution to  $(\ref{fbcvfinal})$.
There exists a universal constant $\bar \varepsilon>0$ such that if $u$
satisfies
\begin{equation*}
u(x) \geq U_\beta(x), \quad \text{in $B_1$}
\end{equation*}
for $0 < \beta \leq L$ and for $0\leq\varepsilon\leq\bar{\varepsilon}$,
\begin{equation}  \label{nondvar}
\|f\|_{L^\infty(B_1)} \leq \varepsilon^2 \min\{G_0(\beta),\beta\},
\quad
\|\mathbf{b}\|_{L^\infty(B_1)} \leq \varepsilon^2,
\end{equation}
\begin{equation}  \label{noncvar}
\|G(\eta,x) - G_0(\eta)\|_{L^\infty(B_1)}\leq \varepsilon^2, \quad \forall 0\leq \eta \leq L,
\end{equation}
then, if at $\bar x=\dfrac{1}{5}e_n$
\begin{equation}  \label{u-p>ep2var}
u(\bar x) \geq U_\beta(\bar x_n + \varepsilon),
\end{equation}
then
\begin{equation}
u(x) \geq U_\beta(x_n+c\varepsilon) \quad \text{in $\overline{B}_{1/2},$}
\end{equation}
for some $0<c<1$ universal. Analogously, if
\begin{equation*}
u(x) \leq U_\beta(x) \quad \text{in $B_1$}
\end{equation*}
and
\begin{equation*}
u(\bar x) \leq U_\beta(\bar x_n - \varepsilon)
\end{equation*}
then
\begin{equation*}
u(x) \leq U_\beta(x_n - c \varepsilon) \quad \text{in $\overline{B}_{1/2}.$}
\end{equation*}
\end{lem}

\begin{proof} We argue as in the proof of Lemma \ref{main} and we only point out the main differences.

By our assumptions, in $B_{1/10}(\bar x) \subset B_1^+(u)$, $u-U_\beta \geq 0$ solves
\begin{equation*}
\mathcal{L} (u-U_\beta)=f-\alpha b_n.
\end{equation*}
Recall that $\alpha = G_0(\beta).$ By Harnack inequality, we  obtain in $\overline B_{1/20}(\bar x)$
\begin{align*}  \label{HInewvar}
u(x) - U_\beta(x) &\geq c(u(\bar x)- U_\beta(\bar x)) - C \|f-\alpha
b_n\|_{L^\infty} \\ &\geq c(u(\bar x)- U_\beta(\bar x)) - C( \|f\|_{L^\infty}+
\alpha\|b\|_{L^\infty}).
\end{align*}

From \eqref{nondvar}, \eqref{u-p>ep2var}  and the inequality above we conclude that for $\varepsilon$ small enough,\begin{equation}  \label{u-p>cepvar}
u - U_\beta \geq \alpha c\varepsilon - \alpha C\varepsilon^2
\geq c_0\alpha \varepsilon \quad \text{in $\overline B_{1/20}(\bar x)$}.
\end{equation}

From \eqref{u-p>cepvar} and the comparison principle it follows that for $c_1$ small universal
\begin{equation}\label{c1}u -\alpha x_n \geq \alpha c_1 \eps x_n, \quad x \in \{x_n>0\} \cap \overline B_{19/20}.\ee
To prove this claim, let $\phi$ solve $$\mathcal L \phi = 0 \quad \text{in $R:= (B_1 \cap \{x_n >0\})\setminus \overline B_{1/20}(\bar x)$}$$ with boundary data
$$\phi=0 \quad \text{on $\p (B_1 \cap \{x_n >0\})$}, \quad \phi=1 \quad \text{on $\p  B_{1/20}(\bar x).$}$$ Then, by boundary Harnack $$\phi \geq c x_n \quad \text{in $\bar R \cap B_{19/20}.$}$$
We now compare $u - \alpha x_n$ with $\frac {1}{2}\alpha c_0 \phi \eps - 8\alpha \eps^2 x_n + 4 \alpha \ep^2 x_n^2$ in the domain $R$ to obtain the desired conclusion.

We now proceed similarly as in Lemma \ref{main},  with $w$ the function defined in \eqref{w}. We compute
\begin{equation*}
\begin{split}
&\sum a_{ij}D_{ij}w(x) \\
&=\gamma(\gamma+2)\mid x-\bar x\mid^{-\gamma-4}\mbox{Tr}(A(x-\bar{x})\otimes
(x-\bar{x}))-\gamma \mid x-\bar x\mid^{-\gamma-2}\mbox{Tr}(A) \\
&\geq \gamma(\gamma+2)\mid x-\bar x\mid^{-\gamma-2}n\lambda-\gamma \mid
x-\bar x\mid^{-\gamma-2}n\Lambda \\
&=\gamma\mid x-\bar x\mid^{-\gamma-2}n\left((\gamma+2)\lambda -\Lambda
\right).
\end{split}%
\end{equation*}

Then
\begin{equation*}  \label{modifiedcheck}
\begin{split}
\mathcal{L}w &\geq \gamma\mid x-\bar
x\mid^{-\gamma-2}n\left((\gamma+2)\lambda -\Lambda \right)-\gamma\mid\mid
\mathbf{b}\mid\mid_{L^{\infty}}\mid x-\bar x\mid^{-\gamma-1} \\
&=\gamma\mid x-\bar x\mid^{-\gamma-2}\left(n\left((\gamma+2)\lambda -\Lambda
\right)-\mid\mid \mathbf{b}\mid\mid_{L^{\infty}}\mid x-\bar x\mid\right) \\
&\geq \gamma\mid x-\bar x\mid^{-\gamma-2}\left(n\left((\gamma+2)\lambda
-\Lambda \right)-\mid\mid \mathbf{b}\mid\mid_{L^{\infty}}\right)\equiv
k_0(\gamma,c_0,n,\lambda,\Lambda)>0,
\end{split}%
\end{equation*}
as long as $\gamma$ satisfies
\begin{equation*}
n\left((\gamma+2)\lambda -\Lambda \right)-\mid\mid \mathbf{b}%
\mid\mid_{L^{\infty}}>0.
\end{equation*}

 Now set $\psi =1-w$ and for $x \in
\overline B_{3/4}(\bar x)
$ define
\begin{equation*}
v_t(x)= \alpha(1+c_1 \eps)(x_n - \varepsilon c_0 \delta \psi(x)+t\varepsilon)^+ -  \beta(x_n - \varepsilon c_0 \delta \psi(x)+t\varepsilon)^-, \end{equation*} with $\delta>0$ small to be made precise later, and $c_1$ the constant in \eqref{c1}.

Then, for $t=-c_1$ one can easily verify that $$v_{-c_1} \leq U_\beta \leq u, \quad  x \in
\overline B_{3/4}(\bar x).$$

Let $\bar t$ be the largest $t \geq -c_1$ such that
\begin{equation*}
v_{t}(x) \leq u(x) \quad \text{in $\overline B_{3/4}(\bar x)$},
\end{equation*} and let $\tilde x$ be the first touching point. To guarantee that $\tilde x$ cannot belong to $\p B_{3/4}$ when $\bar t < c_0 \delta$ we use \eqref{c1}. Indeed if $x \in \p B_{3/4} $ and $v_{\bar t}(x)\geq 0$ then $x_n >0$ and in view of \eqref{c1}
$$v_{\bar t}(x)= \alpha(1+c_1 \eps)(x_n -\eps c_0\delta + \bar t \eps) <  \alpha(1+c_1 \eps)x_n \leq u(x).$$ If $v_{\bar t}(x) <0$ we use that $u \geq U_\beta$ to reach again the conclusion that $v_{\bar t}(x) < u(x)$.
To proceed as in Lemma \ref{main} we now need to show that for $\bar t < c_0 \delta$, $v_{\bar t}$ is a strict subsolution in the annulus $A$.

Indeed, in $A^+(v_{\bar t})$ in view of the assumption \eqref{nondvar} and the computation above for $\mathcal{L}w$, we have
\begin{equation*}
\mathcal{L} v_{\bar t}\geq
\alpha (\varepsilon c_0 \delta k_0+b_n)
\geq
\varepsilon^2\min\{\alpha,\beta\} \geq \|f\|_{\infty}.
\end{equation*}
A similar estimate holds in $A^-(v_{\bar t}).$ Thus
\begin{equation*}
\label{laplacevvar} \mathcal{L} v_{\bar t}\geq f\, \quad \text{in $A^+(v_{\bar t}) \cup A^-(v_{\bar t})$}
\end{equation*}
for $\varepsilon$ small enough.

Also, since $\psi_n< -c$ on $F(v_{\bar t}) \cap A,$ for $\varepsilon$ small,
we have
\begin{equation*}
\kappa \equiv|e_n-\varepsilon c_0\nabla\psi|=(1-2\varepsilon c_0 \delta
\psi_n+\varepsilon^2c_0^2\delta^2|\nabla\psi|^2)^{1/2} = 1+ \tilde k \delta \eps,
\end{equation*} with $\tilde k$ between two universal constants.

Then, on $F(v_{\bar t}) \cap A,$ using \eqref{noncvar},  we can write, as long as $\ep$ is sufficiently small,
\begin{align*}
(v_{\bar{t}}^{+})_{\nu }-G((v_{\bar{t}}^{-})_{\nu },x)&=
\alpha (1+c_1\varepsilon)\kappa -G(\beta \kappa ,x) \geq \alpha (1+c_1\varepsilon
)\kappa -G_0(\beta \kappa)-\epsilon^2\\
&>(1+c_1 \varepsilon) G_0(\beta)-G_0(\beta)\kappa^N-\epsilon^2
\\
&\geq \varepsilon G_0(\beta)(\frac{c_1}{2}-N\tilde k \delta)> 0
\end{align*}
if $\delta< c_1/(2N\tilde\kappa)$.
We used that $G_0(\beta)\geq G_0(0)>0$ and that $G_0(\beta\kappa)<G_0(\beta)\kappa^N$, since $\eta ^{-N}G_{0}(\eta )$ is strictly decreasing.

 Thus, $v_{\bar t}$ is a
strict subsolution to $\eqref{fb}$ in $A$ as desired.
Hence $\bar t \geq c_0 \delta$ and we conclude as in the Laplacian case.

\end{proof}

With Lemma \ref{mainvar} at hand, Harnack Inequality and its Corollary
follow as in Section \ref{Harnacksec}. We only state the Corollary, since it is indeed the tool used in the proof of the improvement of flatness lemma in the next subsection.

\begin{cor} \label{cor_cv}Let $u$ satisfies at some point $x_0 \in B_2$
\be\label{osc_cv} U_\beta(x_n+ a_0) \leq u(x) \leq U_\beta(x_n+ b_0) \quad
\text{in $B_1(x_0) \subset B_2,$}\ee for some $0< \beta \leq L$, with
$$b_0 - a_0 \leq \ep, $$ and let \eqref{nondvar}-\eqref{noncvar} hold, for $\ep \leq \bar \ep,$ $\bar \ep$ universal. Then  in $B_1(x_0)$, ($\alpha=G_0(\beta)$)$$\tilde u_\ep(x) = \begin{cases} \dfrac{u(x) -\alpha x_n }{\alpha\ep}  \quad \text{in $B_2^+(u) \cup F(u)$} \\ \ \\ \dfrac{u(x) -\beta x_n }{\beta\ep}  \quad \text{in $B_2^-(u)$} \end{cases}$$  has a H\"older modulus of continuity at $x_0$, outside
the ball of radius $\ep/\bar \ep,$ i.e for all $x \in B_1(x_0)$, with $|x-x_0| \geq \ep/\bar\ep$
$$|\tilde u_\ep(x) - \tilde u_\ep (x_0)| \leq C |x-x_0|^\gamma.
$$
\end{cor}

\subsection{Improvement of flatness.}

We now extend the basic induction step towards $C^{1,\gamma}$ regularity at $%
0$. We argue as in the proof of Lemma \ref{improv1}.

\begin{lem}
\label{improv1_var}Let $u$ be solution of $(\ref{fbcv})$ and suppose that
\begin{equation}  \label{flat_1_var}
U_\beta(x_n -\varepsilon) \leq u(x) \leq U_\beta(x_n + \varepsilon) \quad
\text{in $B_1,$}
\end{equation}
with $0 < \beta \leq L,$
\begin{equation*}
\|a_{ij}-\delta_{ij}\|_{L^\infty(B_1)}\leq \varepsilon,
\quad
\|f\|_{L^\infty(B_1)} \leq \varepsilon^2\min\{G_0(\beta),\beta)\},
\quad
\|\mathbf{b}\|_{L^\infty(B_1)}\leq \varepsilon^2,
\end{equation*} and
\begin{equation*}  \label{noncvar*}
\|G(\eta,\cdot) - G_0(\eta)\|_{L^\infty(B_1)}\leq \varepsilon^2, \quad \forall 0 \leq \eta \leq L.
\end{equation*}

If $0<r \leq r_0$ for $r_0$ universal, and $0<\varepsilon \leq \varepsilon_0$
for some $\varepsilon_0$ depending on $r$, then

\begin{equation}  \label{improvedflat_2_new_var}
U_{\beta^{\prime}}(x \cdot \nu_1 -r\frac{\varepsilon}{2}) \leq u(x) \leq
U_{\beta^{\prime}}(x \cdot \nu_1 +r\frac{\varepsilon }{2}) \quad \text{in $%
B_r,$}
\end{equation}
with $|\nu_1|=1,$ $|\nu_1 - e_n| \leq \tilde C\varepsilon$, and $|\beta
-\beta^{\prime}| \leq \tilde C\beta \varepsilon$ for a universal constant $%
\tilde C.$
\end{lem}

\begin{proof}
We divide the proof in 3 steps.

\vspace{2mm}

\textbf{Step 1 -- Compactness.} \smallskip
We keep the same notation of  Lemma \ref{improv1}. In this case, the sequence  $u_k$ is
a solution of problem (\ref{fbcv}) for operators
\begin{equation*}
\mathcal{L}^k = \sum_{ij}{a_{ij}}^kD_{ij} + \mathbf{b}^k\cdot \nabla
\end{equation*} with  ($\alpha_k=G_k(\beta_k,0)$)\begin{equation*}
\|a^k_{ij}-\delta_{ij}\|_{L^\infty}\leq \varepsilon_k, \quad
\|f_k\|_{L^\infty} \leq \varepsilon_k^2\min\{\alpha_k,\beta_k\}, \quad
\|\mathbf{b}^k\|_{L^\infty}\leq \varepsilon_k^2,
\end{equation*} and
\begin{equation}  \label{noncvar**}
\|G_k(\eta,\cdot) - G_k(\eta,0)\|_{\infty}\leq \varepsilon_k^2, \quad \forall 0 \leq \eta \leq L.
\end{equation}

The normalized functions $\tilde{u}_k$ are
defined by the same formula.
Up to  a subsequence, $G_k(\cdot,0)$ converges, locally uniformly, to some $C^1$-function $\tilde G_0$, while
$\beta_k\to \tilde \beta$ so that $\alpha_k\to \tilde \alpha
=\tilde G_0(\tilde{\beta}).$
Moreover, by Corollary \ref{cor_cv} the graphs of $\tilde{u}_k$ converge in the Hausdorff distance to a H\"{o}lder continuous $\tilde u.$

\vspace{2mm}

\textbf{Step 2 -- Limiting Solution.}
We show that $\tilde u$ solves
\begin{equation}  \label{Neumann_var}
\begin{cases}
\Delta \tilde u=0 & \text{in $B_{1/2} \cap \{x_n \neq 0\}$}, \\
\  &  \\
\tilde \alpha \tilde u_n^+ - \tilde{\beta}\tilde G_0^{\prime}(\tilde\beta) \tilde u_n^-=0 & \text{%
on $B_{1/2} \cap \{x_n =0\}$}.%
\end{cases}%
\end{equation}
\smallskip

We can write, say in $%
\Omega^+(u^k),$ (in  $\Omega^-(u^k)$ replace $\alpha_k$ with $\beta_k$)
\begin{equation*}
\sum a^k_{ij}D_{ij}\tilde{u}_k=\frac{1}{\alpha_k\varepsilon_k} \sum a_{ij}^{k} D_{ij}u_k=%
\frac{1}{\alpha_k\varepsilon_k} ( -\alpha_k\mathbf{b}^k\cdot\nabla u_k+f^k)\equiv
F^k,
\end{equation*}
where $|F^k|\leq C\varepsilon_k.$

Thus
\begin{equation*}
\Delta\tilde{u}_k=
\sum_{i,j=1}^n(\delta_{ij}-a_{ij}^k)  D_{ij}\tilde u_k+F^k.
\end{equation*}
Hence recalling that $\| a_{ij}^k - \delta_{ij} \|_\infty \leq \varepsilon_k,$ and
from interior $L^p$ Schauder estimates for second derivatives, we conclude
that, for instance, $\Delta\tilde{u}_k\to 0$ in $L^p$ on every compact set contained
in $\Omega^+(\tilde{u}^k)$ or in $\Omega^-(\tilde{u}^k).$  This shows that $%
\tilde{u}$ is harmonic in $B_{1/2} \cap \{x_n \neq 0\}$.

Next, we prove that $\tilde u$ satisfies the transmission condition in %
\eqref{Neumann_var} in the viscosity sense.

Again we argue by contradiction. Let $\tilde \phi$ be a function of the form
\begin{equation*}
\tilde \phi(x) = A+ px_n^+- qx_n^- + B Q(x-y)
\end{equation*}
with
\begin{equation*}
Q(x) = \frac 1 2 [(n-1)x_n^2 - |x^{\prime}|^2], \quad y=(y^{\prime},0),
\quad A \in , B >0
\end{equation*}
and
\begin{equation*}
\tilde{\alpha} p- \tilde \beta \tilde G_0^{\prime}(\tilde{\beta}) q>0,
\end{equation*} and assume  that $\tilde \phi$ 
touches $u$ strictly from below at a point $x_0= (x_0^{\prime}, 0) \in B_{1/2}$. As in Lemma \ref{improv1}, let
\begin{equation*}
\phi_k=a_k\Gamma_k^+(x)-b_k\Gamma_k^-(x)+\alpha_k(d_k^+(x))^2%
\varepsilon_k^{3/2}+\beta_k(d_k^-(x))^2\varepsilon_k^{3/2},
\end{equation*}
where, we recall,
\begin{equation*}
a_k=\alpha_k(1+\varepsilon_k p), \quad b_k=\beta_k(1+\varepsilon_k q)
\end{equation*}
and $d_k(x)$ is the signed distance from $x$ to $\partial B_{\frac{1}{%
B\varepsilon_k}}(y+e_n(\frac{1}{B\varepsilon_k}-A\varepsilon_k)).$ Moreover,
\begin{equation*}
 \psi_k(x)=\phi_k(x+\varepsilon_k c_ke_n)
\end{equation*}
touches $u_k$ from below at $x_k,$ with $c_k \to 0, x_k \to x_0.$

We get a contradiction if we prove that $\psi_k$ is a strict subsolution to
our free boundary problem, that is
\begin{equation*}  \label{fbpsi}
\left \{
\begin{array}{ll}
\mathcal{L}^k \psi_k > f_k, & \hbox{in
$B_1^+(\psi_k) \cup B_1^-(\psi_k),$} \\
\  &  \\
(\psi_k^+)_\nu - G_k((\psi_k^-)_\nu,x) >0, & \hbox{on $F(\psi_k)$.} \\
&
\end{array}
\right.
\end{equation*}

We have
\begin{equation*}
|\nabla\Gamma_k|\leq C, |D_{ij}\Gamma_k|\leq C\varepsilon_k
\end{equation*}
and $|a_{ij}-\delta_{ij}|\leq \varepsilon_k$. We can write, $k$ large
enough, say, in the positive phase of $\psi_k$, 
\begin{align*}
\mathcal L_k \psi_k&=(\mathcal{L}^k-\Delta) \psi_k+\Delta \psi_k
\geq
-C\alpha_k\varepsilon_k^2 + \alpha_k\varepsilon_k^{3/2}\mathcal{L}^kd_k^2(x+\varepsilon c_k e_n)\\
&\geq c\min\{\alpha_k,\beta_k\}\varepsilon_k^{3/2}\geq \|f_k\|_{L^\infty}
\end{align*}
and the first condition is satisfied. An analogous estimate holds in the
negative phase.

Finally, since on the zero level set $|\nabla \Gamma_k|=1$ and $|\nabla
d^2_k|=0$ the free boundary condition reduces to showing that
\begin{equation*}
a_k-G_k(b_k,x) >0.
\end{equation*}
Using the definition of $a_k, b_k$ we need to check that
\begin{equation*}
\alpha_k(1+\varepsilon_kp) -G_k(\beta_k(1+\varepsilon_k q),x) >0.
\end{equation*} From \eqref{noncvar**}, it suffices to check that
\begin{equation*}
\alpha_k(1+\varepsilon_kp) -G_k(\beta_k(1+\varepsilon_k q),0) -\ep_k^2 >0.
\end{equation*}
This inequality holds for $k$ large in view of the fact that
\begin{equation*}
\tilde\alpha p -\tilde\beta \tilde G^{\prime}_0(\tilde\beta) q >0.
\end{equation*}

Thus $\tilde u$ is a viscosity solution to the linearized problem.

\vspace{2mm}

\textbf{Step 3 -- Contradiction.} According to estimate \eqref{lr}, since $%
\tilde u(0)=0$ we obtain that

\begin{equation*}
|\tilde u - (x^{\prime}\cdot \nu^{\prime}+ px_n^+ -qx_n^-)| \leq C r^2,
\quad x\in B_r,
\end{equation*}
with
\begin{equation*}
\tilde \alpha p -\tilde \beta \tilde G_0^{\prime}(\tilde{\beta}) q=0, \quad
|\nu^{\prime}| = |\nabla_{x^{\prime}} \tilde u (0)| \leq C.
\end{equation*}

Thus, since $\tilde u_k$ converges uniformly to $\tilde u$ (by slightly
enlarging $C$) we get that

\begin{equation*}  \label{ukest_var}
|\tilde u_k - (x^{\prime}\cdot \nu^{\prime}+ px_n^+ -qx_n^-)| \leq C r^2,
\quad x\in B_r.
\end{equation*}

Now set,

\begin{equation*}
\beta^{\prime}_k = \beta_k(1+\varepsilon_k q), \quad \nu_k = \frac{1}{%
\sqrt {1+ \varepsilon_k^2 |\nu^{\prime}|^2}}(e_n + \varepsilon_k
(\nu^{\prime}, 0)).
\end{equation*}

Then,
\begin{align*}
\alpha^{\prime}_k &= G_k(\beta_k(1+\varepsilon_k q),0)
=G_k(\beta_k,0)+\beta_k G_k^{\prime}(\beta_k,0)\varepsilon_k q + O(\varepsilon_k^2)\\
&=\alpha_k(1+\beta_k\frac{G_k^{\prime}(\beta_k,0)}{\alpha_k}q\varepsilon_k)+
O(\varepsilon_k^2) =\alpha_k(1+\varepsilon_kp)+ O(\varepsilon_k^2)
\end{align*}
since from the identity  $\tilde \alpha p -\tilde{\beta}\tilde G_0^{\prime}(\tilde{\beta})q=0$ we derive that
\begin{equation*}
\beta_k\frac{G_k^{\prime}(\beta_k,0)}{\alpha_k}q = p + O(\varepsilon_k).
\end{equation*}
Moreover
\begin{equation*}
\nu_k =e_n + \varepsilon_k (\nu^{\prime}, 0) + \varepsilon_k^2 \tau, \quad
|\tau|\leq C.
\end{equation*}

With these choices it follows as in Lemma \ref{improv1} that (for $k$ large and $r \leq r_0$)

\begin{equation*}
\widetilde{U}_{\beta^{\prime}_k}(x\cdot \nu_k -\varepsilon_k\frac r 2) \leq
\tilde u_k(x) \leq \widetilde{U}_{\beta^{\prime}_k}(x\cdot \nu_k
+\varepsilon_k\frac r 2), \quad \text{in $B_r$}
\end{equation*}
which leads to a contradiction.
\end{proof}

\section{The degenerate case for general free boundary problems.}\label{section10}

In this section, we recover the improvement of flatness lemma in the degenerate case, that is when the negative part of $u$ is negligible and
the positive part is close to a one-plane solution (i.e. $\beta=0,\alpha=G_0(0)$). First we need the Harnack inequality.

\subsection {Harnack inequality.}

As in Section \ref{Harnacksec}, Harnack inequality in the degenerate case follows from the
following basic lemma.

\begin{lem}\label{main2cv}There exists
a universal constant $\bar \ep>0$ such that if $u$ satisfies \begin{equation*} u^+(x) \geq U_0(x), \quad \text{in $B_1$}\end{equation*}
with
\be\label{dgcv} \|u^-\|_{L^\infty} \leq \ep^2, \quad \|{\bf b}\|_{L^\infty} \leq \ep^2 \quad \|f\|_{L^\infty} \leq \ep^4,\ee \be\label{cgcv}\|G(\eta,\cdot)- G_0(\eta)\|\leq \ep^2, \quad  0 \leq \eta \leq C\ep^2\ee
 then if at $\bar x=\dfrac{1}{5}e_n$ \be\label{u-p>ep2dcv}
u^+(\bar x) \geq U_0(\bar x_n + \ep), \ee
then \be u^+(x) \geq
U_0(x_n+c\eps) \quad \text{in $\overline{B}_{1/2},$}\ee
for some
$0<c<1$ universal. Analogously, if $$u^+(x) \leq U_0(x) \quad \text{in $B_1$}$$ and $$ u^+(\bar x) \leq U_0(\bar x_n - \eps)$$ then $$ u^+(x) \leq U_0(x_n - c \ep) \quad
\text{in $\overline{B}_{1/2}.$}$$
\end{lem}
\begin{proof} The proof is the same as for the model case in Lemma \ref{main2}. To prove that
$$v_{\bar t}(x)= G_0(0)(x_n - \ep c_0 \psi +  \bar t \ep)^+ -   \ep^2 C_1(x_n - \ep c_0 \psi(x) + \bar t \ep)^-, \quad x \in \overline B_{3/4}(\bar x)
$$ is a subsolution in the annulus $A$ we use the following computation
 \begin{align*}
\mathcal{L} v_{\bar t}&\geq c_0C_1\varepsilon^3 \mathcal{L}w - C_1\varepsilon^2 |b_n| \geq \ep^3 K(n,\lambda, \Lambda) > \ep^4 \geq \|f\|_{\infty}, \quad \textrm{in
$A^+(v_{\bar t}) \cup A^-(v_{\bar t})$}
\end{align*} for $\ep$ small enough. Here we have used as in Lemma \ref{mainvar} that $\mathcal Lw \geq k_0>0.$

Moreover, on $F(v_{\bar t}) \cap A$ we have $$
(v_{\bar t}^+)_\nu - G((v_{\bar t}^-)_\nu)=
G_0(0)|e_n-\ep c_0 \nabla \psi|-G(\ep^2 C_1|e_n-\ep c_0 \nabla \psi|,x) \geq C \ep | \psi_n|+O(\ep^2)  > 0$$ as long as $\ep$ is small enough.
\end{proof}

We state here the Corollary that can be deduced by the degenerate Harnack Inequality.

\begin{cor} \label{corollary4cv}Let $u$ satisfies at some point $x_0 \in B_2$
\be\label{osc_cv_dg} U_0(x_n+ a_0) \leq u(x) \leq U_0(x_n+ b_0) \quad
\text{in $B_1(x_0) \subset B_2,$}\ee
with
$$b_0 - a_0 \leq \ep, $$ and let \eqref{dgcv}-\eqref{cgcv} hold
for $\ep \leq \bar \ep,$ $\bar \ep$ universal. Then  in $B_1(x_0)$
$$\tilde
u_\ep:= \frac{u^+(x) - G_0(0)x_n}{\ep G_0(0)}$$ has a H\"older modulus of continuity at $x_0$, outside
the ball of radius $\ep/\bar \ep,$ i.e for all $x \in B_1(x_0)$, with $|x-x_0| \geq \ep/\bar\ep$
$$|\tilde u_\ep(x) - \tilde u_\ep (x_0)| \leq C |x-x_0|^\gamma.
$$
\end{cor}

\subsection{Improvement of flatness} We prove here the improvement of flatness in the degenerate setting.
Recall that in this case one improves the flatness of $u^+$ only.

\begin{lem}
\label{improv4}Let $u$ satisfy
\begin{equation}  \label{flat_var}
U_0(x_n -\varepsilon) \leq u^+(x) \leq U_0(x_n + \varepsilon) \quad \text{in $B_1,$} \quad 0\in F(u),
\end{equation}
with
\begin{equation*}
\|a_{ij}-\delta_{ij}\|\leq \varepsilon,
\quad
\|f\|_{L^\infty(B_1)} \leq \varepsilon^4,
\quad
\|\mathbf{b}\|_{L^\infty(B_1)} \leq \varepsilon^2,
\end{equation*}
\begin{equation*}
\|G(\eta,\cdot) - G_0(\eta)\|_{L^\infty} \leq \ep^2, \quad 0 \leq \eta \leq C\ep^2,
\end{equation*}
and
\begin{equation*}
\|u^-\|_{L^\infty(B_1)} \leq \varepsilon^2.
\end{equation*}

If $0<r \leq r_1$ for $r_1$ universal, and $0<\varepsilon \leq \varepsilon_1$
for some $\varepsilon_1$ depending on $r$, then
\begin{equation}  \label{improvedflat_2_var}
U_0(x \cdot \nu_1 -r\frac{\varepsilon}{2}) \leq u^+(x) \leq U_0(x \cdot
\nu_1 +r\frac{\varepsilon }{2}) \quad \text{in $B_r,$}
\end{equation}
with $|\nu_1|=1,$ $|\nu_1 - e_n| \leq C\varepsilon$ for a universal constant
$C.$
\end{lem}

\begin{proof}

\textbf{Step 1 -- Compactness.} As in Lemma \ref{improv4_deg}, it follows from Corollary \ref{corollary4cv} that as $\ep_k \rightarrow 0$ the graphs of the
\begin{equation*}
\tilde{u}_{k}(x)= \dfrac{u_k(x) -G_k(0,0)x_n}{G_k(0,0) \varepsilon_k}, \quad x \in
B_1^+(u_k) \cup F(u_k)
\end{equation*}converge (up to a
subsequence) in the Hausdorff distance to the graph of a H\"older
continuous function $\tilde{u}$ over $B_{1/2} \cap \{x_n \geq 0\}$. Here the $u_k$ solve our free boundary problem \eqref{fbcv} with coefficients $a_{ij}^k, {\bf b}^k$, right-hand-side $f_k$ and free boundary condition $G_k$ satisfying the assumptions of the lemma for a sequence of $\ep_k$'s going to 0.

\smallskip

\textbf{Step 2 -- Limiting Solution.} One shows that $\tilde u$
solves the following Neumann problem
\begin{equation}  \label{Neumann4_var}
\begin{cases}
\Delta \tilde u=0 & \text{in $B_{1/2} \cap \{x_n > 0\}$}, \\
\  &  \\
\tilde u_n =0 & \text{on $B_{1/2} \cap \{x_n =0\}$}.%
\end{cases}%
\end{equation}

We can easily adapt the proof of Lemma \ref{improv4_deg}, choosing 
$$\phi_k(x)= a_k \Gamma^+_k(x)+ (d_k^+(x))^2\ep_k^{3/2}, \quad a_k=G_k(0,0)(1+\ep_k p).$$
and 
 \begin{equation}\Psi_k(x) = \begin{cases}\phi_k(x + c_k\eps_k e_n) & \text{in $\mathcal B$}\\  \ & \ \\ c\eps^2_k (3d(x,\p \mathcal B) + d^2(x,\p \mathcal B)) & \text{outside of $\mathcal B,$}\end{cases}\ee with $$\mathcal B: =B_{\frac{1}{B\ep_k}}(y+e_n(\frac{1}{B\ep_k}-A\ep_k-\eps_kc_k)).$$ 
To check the subsolution condition at the free boundary for the function $\Psi_k(x)$, we need that 
$$(\Psi_k^+)_\nu>G_k((\Psi_k^-)_\nu,x)  , \quad \hbox{on $F(\Psi_k)$.}
$$

This is equivalent to show that, for $k$ large,

\begin{equation*}
G_k(0,0)(1+\varepsilon_k p) -G_k(c \varepsilon_k^2, x)> 0.
\end{equation*}
Since $p>0$, this follows immediately from the assumptions on $G_k$.
\smallskip

\textbf{Step 3 -- Contradiction.} In this step we can argue as in the final
step of the proof of Lemma 4.1 in \cite{D}.
\end{proof}

\section{Proofs of the main theorems for general free boundary problems} \label{section11}

The proof of Theorem \ref{flatmain2} and Theorem \ref{Lipmainvar} follow the
same scheme of the model case. In particular, for Theorem \ref{flatmain2}, we take care of choosing $\bar{r}^{\bar\gamma}<1/16$, say, while the other assumptions on $\bar r$ remain the same. Also, $\tilde \eps$ may have to be smaller, depending on $\gamma_0$. The dichotomy degenerate/nondegenerate is handled through Lemma \ref{finalcase} which extends to the variable coefficients case, with minor changes in the proof.

In the proof of Theorem \ref{Lipmainvar}, the blow-up limit $\tilde u$ solves the following global homogeneous two-phase free
boundary problem

\begin{equation}  \label{fbglobalvar}
\left \{
\begin{array}{ll}
\Delta \tilde u = 0, & \hbox{in $\{\tilde u>0\} \cup \{\tilde u \leq 0\}^0$}
\\
\  &  \\
\tilde u_\nu^+ =G_0 (\tilde u_\nu^-), & \hbox{on $F(\tilde u):= \p \{\tilde
u>0\},$} \\
&
\end{array}%
\right.
\end{equation}

 Now, Lemma \ref{Liouville} holds with identical proof for the free boundary condition $U_{\nu}^+=G_0(U_{\nu}^-)$, so that the proof of Theorem \ref{Lipmainvar} does not present any further difficulty.

\end{document}